\documentclass[11pt]{siamltex}
\usepackage{amsmath}
\usepackage{bm}
\usepackage{amssymb,version}
\usepackage{cases}
\usepackage{color}
\usepackage{verbatim}

\usepackage{graphicx}
\usepackage{subfigure}
\allowdisplaybreaks
\begin{document}
    \title{Stability and error analysis of IMEX SAV schemes for the magneto-hydrodynamic equations 
\thanks{This work is supported in part  by the National Natural Science Foundation of China  grants  11901489 and 11971407,  NSF grant  DMS-2012585 and AFOSR Grant FA9550-20-1-0309.}}

    \author{ Xiaoli Li\thanks{School of Mathematics, Shandong University, Jinan, Shandong, 250100, P.R. China. Email: xiaolisdu@163.com}.
         \and Weilong Wang\thanks{School of Mathematical Sciences and Fujian Provincial Key Laboratory on Mathematical Modeling and High Performance Scientific Computing, Xiamen University, Xiamen, Fujian, 361005, China. Email: wwlmath@foxmail.com}.                
        \and Jie Shen\thanks{Corresponding Author. Department of Mathematics, Purdue University, West Lafayette, IN 47907, USA. Email: shen7@purdue.edu}.
}

\maketitle

\begin{abstract}
We construct and analyze first- and second-order implicit-explicit (IMEX) schemes based on the scalar auxiliary variable (SAV) approach for the magneto-hydrodynamic equations. These schemes are linear,  only require solving a sequence of linear differential equations with constant coefficients at each time step, and  are unconditionally  energy stable. We  derive rigorous error estimates for the velocity, pressure and magnetic field of the first-order scheme in the two dimensional case without any condition on the time step. Numerical examples are presented to validate the proposed schemes. 
\end{abstract}

\markboth{XIAOLI LI, JIE SHEN AND WEILONG WANG} {IMEX SAV schemes for the MHD equations}
 \section{Introduction}
We consider in this paper numerical approximation of the following magneto-hydrodynamic (MHD) equations \cite{sermange1983some}: 
  \begin{subequations}\label{e_MHD_model}
    \begin{align}
     \frac{\partial \textbf{u} }{\partial t} + ( \textbf{u} \cdot \nabla ) \textbf{u} 
      - \nu \Delta \textbf{u} + \nabla p - \alpha ( \nabla \times \textbf{b} ) \times \textbf{b} = 0
     \quad &\  \rm{in} \ \Omega\times J,    \label{e_MHD_modelA} \\
   \frac{\partial \textbf{b} }{\partial t} + \eta \nabla \times ( \nabla \times \textbf{b} ) + \nabla \times ( \textbf{b} \times \textbf{u} ) = 0       \quad &\  \rm{in} \ \Omega\times J,    \label{e_MHD_modelB} \\ 
      \nabla\cdot\textbf{u}=0 , \  \nabla\cdot\textbf{b}=0
      \quad &\ \rm{in} \ \Omega\times J,  \label{e_MHD_modelC} 
      \end{align}
       \end{subequations}
 with boundary and initial conditions
  \begin{align*}
  &   \textbf{u}=\textbf{0}, \ \ \textbf{b} \cdot \textbf{n} =0, \ \ \textbf{n} \times  ( \nabla \times \textbf{b} ) =0  \quad  \rm{on} \ \partial\Omega\times J, \\
  &   \textbf{u} ( \textbf{x}, 0 ) = \textbf{u}^0 ( \textbf{x} ), \quad \textbf{b} ( \textbf{x}, 0 ) = \textbf{b}^0 ( \textbf{V} ) \quad  \rm{in} \ \Omega,
    \end{align*}
 where $\Omega$ is an open bounded domain in $\mathbb{R}^d \;(d=2,3)$ with a sufficiently smooth boundary $\partial \Omega$, 
$\textbf{n}$ is the unit outward normal  of the domain $\Omega$, $ J=(0,T] $, $(\textbf{u},p, \textbf{b} )$ represent respectively the unknown  velocity, pressure and magnetic field.  The parameters $ \nu $ and $ \eta $ are kinematic viscosity and magnetic diffusivity, respectively, and $ \alpha = 1/ (4 \pi \mu \rho ) $ with $ \mu $ as the magnetic permeability and $ \rho $ as the fluid density.

The   MHD system is used to describe the interaction between a viscous, incompressible, electrically conducting fluid and an external magnetic field.  When a conducting fluid is placed in an existing magnetic field, the fluid motion produces electric currents which in turn create forces on the fluid and change the magnetic field itself.   It has been widely used in many science and engineering applications, such as liquid metal cooling for nuclear reactors, sustained plasma confinement for controlled thermonuclear fusion, etc \cite{goldstein2002classical,davidson2002introduction}. The mathematical theory of MHD equations can be found in \cite{sermange1983some}.

Numerical approximation of the MHD equations is challenging, as it involves delicate nonlinear coupling  between the velocity and magnetic field in addition to the difficulties associated with the Navier-Stokes equations and Maxwell equations. 
There exists a large literature devoted to constructing compatible  spatial discretization for the MHD equations, see  \cite{yee1966numerical,babuvska1971error,nedelec1980mixed,Gerbeau2006Mathematical,brezzi2012mixed} and related references. In this paper, we are only concerned with time discretization, which can be coupled with any well developed compatible spatial discretization. 

The MHD equations \eqref{e_MHD_model} is energy dissipative. More precisely, taking the inner products of \eqref{e_MHD_modelA} and  \eqref{e_MHD_modelB} with $\textbf{u}$ and $\alpha\textbf{b}$, respectively, summing up the results,  we find that the  nonlinear terms  do not contribute to the energy and that  the following energy dissipation law holds:
\begin{equation}\label{engdiss}
 \frac d{dt}E(\textbf{u},\textbf{b})=-\nu\|\nabla \textbf{u}\|^2-\alpha\eta\|\nabla\times \textbf{b}\|^2
\quad\text{with }\; 
 E(\textbf{u},\textbf{b})=\frac 12\|\textbf{u}\|^2+\frac{\alpha}2\|\textbf{b}\|^2.
\end{equation}
It is thus desirable to construct numerical schemes which  satisfy a discrete energy dissipation law.

Most existing work use fully implicit or semi-implicit treatments for the nonlinear terms so that the effect of nonlinear coupling can cancel each other and a  discrete energy dissipation law can be derived. However, one needs to solve a nonlinear system or a coupled linear system with time dependent coefficients at each time step. For examples, 
Armero and Simo developed in \cite{armero1996long} energy dissipative schemes for an abstract evolution equation with applications to the incompressible MHD equations; 
Tone \cite{tone2009long} considered an implicit Euler scheme for the 2D MHD equations and established a uniform H2 stability; Layton et al. constructed in \cite{layton2014numerical} two partitioned methods for uncoupling evolutionary MHD flows; Hiptmair et al. \cite{hiptmair2018fully} developed a fully divergence-free finite element method for MHD equations with a semi-implicit treatment of the nonlinear terms; Zhang et al. \cite{zhang2018second} proposed a second order linear BDF scheme with an extrapolated treatment for the nonlinear terms and proved its unconditionally stability and convergence, cf. also \cite{zhang2015unconditional}; And most recently, Li et al. \cite{li2020convergent} proposed a fully discrete linearized H1 conforming Lagrange finite element method, and derived the convergence based on the reg- ularity of the initial conditions and source terms without extra assumptions on the regularity of the solution. To alleviate the cost of solving fully coupled systems at each time step, Badia et al. \cite{badia2013unconditionally} developed an operator splitting algorithm by a stabilized finite element formulation based on projections; Choi and Shen \cite{choi2016efficient} constructed several efficient splitting schemes based on the standard and rotational pressure-correction schemes with a semi-implicit treatment of the nonlinear terms for the MHD equations.

 From a computational point of view, it is desirable for a numerical scheme  to treat the nonlinear term explicitly while still being energy dissipative, so that one only needs to solve simple linear equations with constant coefficients at each time step. However, with a direct  explicit treatment of the nonlinear terms, their energy contribution no longer vanishes, so it becomes very difficult to derive a uniform bound for the numerical solution.  Liu and Pego \cite{liu2010stable} constructed a first-order scheme with fully explicit treatment of the nonlinear terms and showed that its numerical solution is bounded with the time step sufficiently small, but their scheme is not shown to be energy dissipative. The recently proposed scalar auxiliary variable (SAV) approach \cite{shen2018scalar,shen2018convergence,shen2019new} provides a general approach to construct linear, decoupled  unconditionally energy stable schemes for gradient flows. The approach has been extended to Navier-Stokes equations in  \cite{lin2019numerical}. However, the scheme in   \cite{lin2019numerical} requires solving a nonlinear algebraic equation whose well posedness is not guaranteed.  We introduced in \cite{li2020new} a different SAV approach which leads to purely linear and unconditionally stable schemes for the Navier-Stokes equations, and proved corresponding error estimates. 

 The aim of this work is to extend the approach proposed in  \cite{li2020new} to   the MHD equations which are much more complicated with  nonlinear couplings between the velocity and  magnetic fields. Our main contributions are two-folds:
 \begin{itemize}
\item   We construct first- and second-order IMEX SAV schemes for the MHD equations and show that they are unconditionally energy stable. These schemes only require solving a sequence of differential equations with constant coefficients at each time step so they are very efficient and easy to implement.
\item We establish rigorous error estimates for the first-order scheme in the two-dimensional case without any condition on the time step.  
 \end{itemize}
 Compared to the Navier-Stokes equations or Maxwell's equations, the error analysis  for the MHD equations is much more involved due to the nonlinear coupling terms. Our error analysis uses essentially the unconditional bounds of the numerical solution that we derive for our SAV schemes. 
 To the best of our knowledge, this is the first linear, unconditional energy stable and convergent schemes with fully explicit treatment of nonlinear terms for the MHD equations.

The paper is organized as follows. In Section 2,  we construct our IMEX SAV schemes and prove their stability. In Section 3,  we carry out a rigorous error analysis for the first-order IMEX SAV scheme in the two-dimensional case. We present some numerical experiments to validate our schemes in Section 4, and conclude with a few remarks in Section 5. 

  \section{The SAV schemes and their energy stability}
 In this section, we construct first- and second-order IMEX schemes based on the SAV approach  for the MHD equations, and show that they are unconditionally energy stable.
 
 We introduce a scalar auxiliary variable (SAV):
 \begin{equation}\label{e_definition of q}
\aligned
q(t)=\rm{exp} (-\frac{t}{T}),
\endaligned
\end{equation} 
and expand the system \eqref{e_MHD_model} as follows: 
  \begin{numcases}{}
 \frac{\partial \textbf{u}}{\partial t}
     -\nu\Delta\textbf{u}+\nabla p + \exp( \frac{t}{T} ) q(t) ( \textbf{u}\cdot \nabla  \textbf{u}- \alpha  ( \nabla \times \textbf{b} ) \times \textbf{b}) = 0,  \label{e_MHD_model_SAV1} \\
   \frac{\partial \textbf{b} }{\partial t} + \eta \nabla \times ( \nabla \times \textbf{b} ) +  \exp( \frac{t}{T} ) q(t) \nabla \times ( \textbf{b} \times \textbf{u} ) = 0,    \label{e_MHD_model_SAV2}  \\
      \nabla\cdot\textbf{u}=0 , \  \nabla\cdot\textbf{b}=0, \label{e_MHD_model_SAV3} \\      
  \frac{\rm{d} q}{\rm{d} t}=-\frac{1}{T}q + \exp( \frac{t}{T} ) \big(  (\textbf{u}\cdot \nabla  \textbf{u}, \textbf{u} )  - \alpha  \left(  ( \nabla \times \textbf{b} ) \times \textbf{b}, \textbf{u} \right)  
  + \alpha  \left(  \nabla \times ( \textbf{b} \times \textbf{u} ), \textbf{b} \right) \big).     \label{e_MHD_model_SAV4}
\end{numcases}
 Since the sum of the nonlinear terms  in   \eqref{e_MHD_model_SAV4} is zero so \eqref{e_MHD_model_SAV4} is equivalent to the time derivative of \eqref{e_definition of q}.
Hence, with $q(0)=1$, the exact solution of \eqref{e_MHD_model_SAV4} is given by \eqref{e_definition of q}, so
 that  \eqref{e_MHD_model_SAV1}-\eqref{e_MHD_model_SAV3} is exactly the same as \eqref{e_MHD_model}. Therefore, the above system is equivalent to the original system. Note that we  have, in addition to the original energy law \eqref{engdiss}, an additional energy law
 \begin{equation}\label{engdiss2}
 \frac 12\frac d{dt}(\|\textbf{u}\|^2+{\alpha}\|\textbf{b}\|^2+|q|^2) =-\nu\|\nabla \textbf{u}\|^2-\alpha\eta\|\nabla\times \textbf{b}\|^2 -\frac 1T|q|^2.
\end{equation}
Note that, unlike in the original SAV approach, the SAV $q(t)$ is related to the nonlinear part of the free energy, here the SAV $q(t)$ is  pure artificial but will allow us to construct unconditional energy stable, with respect to the energy in \eqref{engdiss2}, schemes with fully explicit treatment of the nonlinear terms.

\subsection{The IMEX SAV schemes}
We set $$\Delta t=T/N,\ t^n=n\Delta t, \ d_t g^{n+1}=\frac{g^{n+1}-g^n}{\Delta t},
\ {\rm for} \ n\leq N.$$

\textbf{Scheme \uppercase\expandafter{\romannumeral 1} (first-order):} Find ($ \textbf{u}^{n+1}, p^{n+1}, q^{n+1},\textbf{b}^{n+1} $) by solving
    \begin{eqnarray}
   && d_t \textbf{u}^{n+1} - \nu \Delta \textbf{u}^{n+1} + \nabla p^{n+1} =  \exp( \frac{t^{n+1} }{T} ) q^{n+1} (\alpha (\nabla \times \textbf{b}^n)  \times \textbf{b}^n - \textbf{u}^{n}\cdot \nabla \textbf{u}^{n}), \label{e_SAV_scheme_first_u} \\
 &&  d_t \textbf{b}^{n+1} + \eta \nabla \times ( \nabla \times \textbf{b}^{n+1} ) +  \exp( \frac{t^{n+1} }{T} ) q^{n+1} \nabla \times ( \textbf{b}^n \times \textbf{u}^n ) = 0, \label{e_SAV_scheme_first_b}\\   
 &&  \nabla\cdot\textbf{u}^{n+1} =0, 
\ \ \ 
\nabla\cdot\textbf{b}^{n+1} =0, \label{e_SAV_scheme_first_div} \\
  &&   \textbf{u}^{n+1} |_{\partial \Omega} =\textbf{0}, \ \ \textbf{b}^{n+1} \cdot \textbf{n} |_{\partial \Omega}=0, \ \ \textbf{n} \times  ( \nabla \times \textbf{b}^{n+1} ) |_{\partial \Omega} =0,  \label{e_SAV_scheme_first_boundary} \\
&&  d_t q^{n+1} = -\frac{1}{T}q^{n+1} + \exp( \frac{t^{n+1} }{T} ) \nonumber\\
&& \big( (\textbf{u}^n\cdot\nabla  \textbf{u}^n, \textbf{u}^{n+1}) - \alpha  (( \nabla \times \textbf{b}^n)  \times \textbf{b}^n, \textbf{u}^{n+1})
+ \alpha(  \nabla \times  (\textbf{b}^n \times \textbf{u}^n)  , \textbf{b}^{n+1} )\big),\label{e_SAV_scheme_first_q} 
     \end{eqnarray} 
     
 We now describe how to solve the semi-discrete-in-time scheme
 \eqref{e_SAV_scheme_first_u}-\eqref{e_SAV_scheme_first_boundary} efficiently.  We denote $S^{n+1}=
\exp ( \frac{t^{n+1}}{T}) q^{n+1} $ and set
  \begin{numcases}{}
  \textbf{b}^{n+1}= \textbf{b}^{n+1}_1 +S^{n+1}\textbf{b}^{n+1}_2,\label{e_split_b} \\
   \textbf{u}^{n+1}=\textbf{u}_1^{n+1}+S^{n+1}\textbf{u}_2^{n+1},\label{e_split_u} \\ 
   p^{n+1}=p_1^{n+1}+S^{n+1}p_2^{n+1}. \label{e_split_p}
\end{numcases}

 Plugging  \eqref{e_split_b}-\eqref{e_split_p} in the scheme  \eqref{e_SAV_scheme_first_u}-\eqref{e_SAV_scheme_first_boundary}, we find  that $\textbf{u}_i^{n+1}, p_i^{n+1} $ $(i=1,2)$ 
 satisfy 
 \begin{numcases}{}
\frac{ \textbf{u}_1^{n+1}-\textbf{u}^{n}}{\Delta t}= \nu\Delta \textbf{u}_1^{n+1}-\nabla p^{n+1}_1 , 
\label{e_implementation_u1p1} \\
\frac{ \textbf{u}_2^{n+1} }{ \Delta t } +  \textbf{u}^{n}\cdot \nabla  \textbf{u}^{n}= \nu\Delta \textbf{u}_2^{n+1} - \nabla p^{n+1}_2 + \alpha ( \nabla \times \textbf{b}^n ) \times \textbf{b}^n, \label{e_implementation_u2p2} \\
 \nabla \cdot  \textbf{u}^{n+1}_i =0,\ \  \textbf{u}^{n+1}_i |_{\partial \Omega} =\textbf{0},\quad i=1,2. 
\label{e_implementation_divu}
\end{numcases}
Next we  determine $\textbf{b}_{i}^{n+1}$ $(i=1,2)$ from 
 \begin{numcases}{}
\frac{ \textbf{b}_1^{n+1}-\textbf{b}^{n}}{\Delta t} + \eta \nabla \times ( \nabla \times \textbf{b}^{n+1}_1 )  = 0 , \label{e_implementation_b1} \\
\frac{ \textbf{b}_2^{n+1} }{ \Delta t } + \eta \nabla \times ( \nabla \times \textbf{b}^{n+1}_2 ) +  \nabla \times ( \textbf{b}^n \times \textbf{u}^n ) = 0,  \label{e_implementation_b2} \\
\nabla\cdot\textbf{b}^{n+1}_i =0, \ \ \textbf{b}^{n+1}_i \cdot \textbf{n} |_{\partial \Omega}=0, \ \ \textbf{n} \times  ( \nabla \times \textbf{b}^{n+1}_i ) |_{\partial \Omega} =0,\quad i=1,2. 
 \label{e_implementation_divb}
\end{numcases}
 Once $\textbf{u}_{i}^{n+1}$, $p_i^{n+1} $, $\textbf{b}_i^{n+1}$ $(i=1,2)$ are known, we can determine explicitly  $S^{n+1}$ from \eqref{e_SAV_scheme_first_q} as follows:
 \begin{equation} \label{e_S_solve}
 \aligned
 \left( \frac{T+\Delta t}{T \Delta t }- \exp( \frac{2 t^{n+1} }{T} ) A_2 \right) \exp( -\frac{t^{n+1} }{T} )S^{n+1} =\exp( \frac{t^{n+1} }{T} )A_1 +\frac{1}{\Delta t} q^n,
 \endaligned
\end{equation}
where 
 \begin{equation*} 
 \aligned
  & A_i=   (\textbf{u}^n\cdot\nabla \textbf{u}^n,\tilde{\textbf{u}}_i^{n+1}) - \alpha \left(  ( \nabla \times \textbf{b}^n ) \times \textbf{b}^n, \textbf{u}_i^{n+1} \right) + \alpha  \left(  \nabla \times ( \textbf{b}^n \times \textbf{u}^n ) , \textbf{b}^{n+1}_i \right), \ i = 1, 2.   
 \endaligned
\end{equation*}
Finally,  we can obtain $\textbf{u}^{n+1}$, $p^{n+1}$ and $\textbf{b}^{n+1}$ from  \eqref{e_split_b}-\eqref{e_split_p}. 

In summary, at each time step, we only need to solve two generalized Stokes equations  in  \eqref{e_implementation_u1p1}-\eqref{e_implementation_divu}, and two elliptic equations \eqref{e_implementation_b1}-\eqref{e_implementation_divb} with constant ciefficients  plus a linear algebraic equation \eqref{e_S_solve} at each time step.  Hence, the scheme is very efficient.

\medskip
\textbf{Scheme \uppercase\expandafter{\romannumeral 2} (second-order):}  
 Find  ($ \textbf{u}^{n+1}, p^{n+1}, q^{n+1},\textbf{b}^{n+1} $) by solving
      \begin{eqnarray}
   && \frac{ 3 \textbf{u}^{n+1}-4\textbf{u}^{n}+\textbf{u}^{n-1} }{ 2 \Delta t }  - \nu \Delta \textbf{u}^{n+1} + \nabla p^{n+1} \nonumber\\
   &&=  \exp( \frac{t^{n+1} }{T} ) q^{n+1}\big( \alpha(\nabla \times \bar{ \textbf{b} }^{n+1} ) \times \bar{ \textbf{b} }^{n+1}-  \bar{ \textbf{u} }^{n+1}\cdot \nabla  \bar{ \textbf{u} }^{n+1}\big), \label{e_SAV_scheme_second_u} \\
 &&  \frac{ 3 \textbf{b}^{n+1}-4\textbf{b}^{n}+\textbf{b}^{n-1} }{ 2 \Delta t } + \eta \nabla \times ( \nabla \times \textbf{b}^{n+1} ) +  \exp( \frac{t^{n+1} }{T} ) q^{n+1} \nabla \times ( \bar{ \textbf{b} }^{n+1} \times \bar{ \textbf{u} }^{n+1} ) = 0, \label{e_SAV_scheme_second_b}  \\
&&  \nabla\cdot\textbf{u}^{n+1} =0 , 
\ \ \ 
\nabla\cdot\textbf{b}^{n+1} =0, \label{e_SAV_scheme_second_div} \\
  &&   \textbf{u}^{n+1} |_{\partial \Omega} =\textbf{0}, \ \ \textbf{b}^{n+1} \cdot \textbf{n} |_{\partial \Omega}=0, \ \ \textbf{n} \times  ( \nabla \times \textbf{b}^{n+1} ) |_{\partial \Omega} =0, \label{e_SAV_scheme_second_boundary}\\   
&&  \frac{ 3q^{n+1}- 4q^n + q^{n-1} }{ 2\Delta t } =  -\frac{1}{T}q^{n+1}+  \exp( \frac{t^{n+1} }{T} ) \nonumber \\
&& \left[ \alpha (  (\nabla \times ( \bar{ \textbf{b} }^{n+1} \times \bar{ \textbf{u} }^{n+1} ) , \textbf{b}^{n+1}) - \alpha(   (\nabla \times \bar{ \textbf{b} }^{n+1})  \times \bar{ \textbf{b} }^{n+1}, \textbf{u}^{n+1} ) + ( \bar{ \textbf{u} }^{n+1} \cdot \nabla  \bar{ \textbf{u} }^{n+1},\textbf{u}^{n+1}) \right], \label{e_SAV_scheme_second_q} 
     \end{eqnarray} 
where $\bar{\textbf{v}}^{n+1}=2\textbf{v}^{n}-\textbf{v}^{n-1}$ for any function $\textbf{v}$. For  $n = 0$, we can compute ($\textbf{u}^{1}$, $p^{1}$, $q^{1}$, $\textbf{b}^{1}$) by the first-order scheme described above.

The second-order scheme \eqref{e_SAV_scheme_second_u}-\eqref{e_SAV_scheme_second_q} can be implemented  the same way as the first-order scheme \eqref{e_SAV_scheme_first_u}-\eqref{e_SAV_scheme_first_q}. 

  \subsection{Energy Stability} 
We  show below that the first- and second-order SAV schemes \eqref{e_SAV_scheme_first_u}-\eqref{e_SAV_scheme_first_q} and \eqref{e_SAV_scheme_second_u}-\eqref{e_SAV_scheme_second_q} are unconditionally energy stable. 
We shall use $\|\cdot\|$  and  $(\cdot, \cdot)$ to denote the norm and inner product in $L^2(\Omega)$,  and $<\cdot, \cdot>$ to denote the inner product in $L^2(\partial\Omega)$. 

 \medskip
 
 \begin{theorem}\label{thm_energy stability_first order}
The scheme \eqref{e_SAV_scheme_first_u}-\eqref{e_SAV_scheme_first_q} is unconditionally stable in the sense that
\begin{equation}\label{discteng}
\aligned
E^{n+1}-E^{n} \leq - \nu\Delta t \| \nabla \textbf{u}^{n+1} \|^2 -  \eta \alpha \Delta t \| \nabla \textbf{b}^{n+1} \|^2 -\frac{1}{T}\Delta t|q^{n+1}|^2, \ \ \forall \Delta t,\; n\geq 0,
\endaligned
\end{equation} 
where 
\begin{equation*}
E^{n+1}=\frac 12 \|\textbf{u}^{n+1}\|^2 + \frac \alpha 2 \|\textbf{b}^{n+1}\|^2 +\frac 12|q^{n+1}|^2 .
\end{equation*} 
\end{theorem}

\begin{proof} 
Taking the inner product of \eqref{e_SAV_scheme_first_u} with   $\Delta t \textbf{u}^{n+1}$ and using the identity 
\begin{equation}\label{e_identity_Euler}
\aligned
(a-b,a)=\frac{1}{2}(|a|^2-|b|^2+|a-b|^2),
\endaligned
\end{equation} 
we have
\begin{equation}\label{e_stability_first_u}
\aligned
&\frac{\| \textbf{u} ^{n+1}\|^2-\| \textbf{u}^{n} \|^2}{2}+\frac{\| \textbf{u}^{n+1}-\textbf{u}^{n} \|^2}{2}
+\nu\Delta t \| \nabla \textbf{u}^{n+1} \|^2 + \Delta t(\nabla p^{n+1}, \textbf{u}^{n+1}) \\
& = \Delta t \exp( \frac{t^{n+1} }{T} ) q^{n+1} \left( \alpha  ( \nabla \times \textbf{b}^n ) \times \textbf{b}^n, \textbf{u}^{n+1} )-\textbf{u}^{n}\cdot \nabla  \textbf{u}^{n},\textbf{u}^{n+1} )\right) .
\endaligned
\end{equation} 
Taking the inner product of \eqref{e_SAV_scheme_first_b} with   $ \alpha \Delta t \textbf{b}^{n+1}$ and using the identity 
\begin{equation} \label{e_curl(curl b)}
\aligned
& \nabla \times ( \nabla \times \textbf{b}^{n+1} ) = -\Delta \textbf{b}^{n+1} + \nabla ( \nabla \cdot \textbf{b}^{n+1} ),
\endaligned
\end{equation}
we have
\begin{equation}\label{e_stability_first_b}
\aligned
& \alpha \frac{\| \textbf{b} ^{n+1}\|^2-\| \textbf{b}^{n} \|^2}{2}+ \alpha \frac{\| \textbf{b}^{n+1}-\textbf{b}^{n} \|^2 }{2} + \eta \alpha \Delta t \| \nabla \textbf{b}^{n+1} \|^2 \\
& \ \ \ \ \ + \alpha \Delta t  \exp( \frac{t^{n+1}}{T} ) q^{n+1} \left(  \nabla \times ( \textbf{b}^n \times \textbf{u}^n ) , \textbf{b}^{n+1} \right) =0. 
\endaligned
\end{equation} 
Multiplying \eqref{e_SAV_scheme_first_q} by $q^{n+1}\Delta t$ leads to
\begin{equation}\label{e_stability_first_q}
\aligned
& \frac{ |q^{n+1}|^2 - |q^n|^2 } { 2 } + \frac{1}{2} | q^{n+1}- q^n |^2 +  \frac{1}{T} \Delta t |q^{n+1}|^2 \\
&=  \Delta t q^{n+1} \exp( \frac{t^{n+1} }{T} )\big(  (\textbf{u}^n\cdot\nabla  \textbf{u}^n, \textbf{u}^{n+1}) 
-  \alpha ( (\nabla \times \textbf{b}^n ) \times \textbf{b}^n, \textbf{u}^{n+1})
+ \alpha (  \nabla \times  (\textbf{b}^n \times \textbf{u}^n) , \textbf{b}^{n+1}) \big) . 
\endaligned
\end{equation}
Then summing up \eqref{e_stability_first_u} with \eqref{e_stability_first_b}-\eqref{e_stability_first_q} results in 
\begin{equation*}\label{e_stability_first_final}
\aligned
& \|\textbf{u}^{n+1}\|^2-\|\textbf{u}^{n}\|^2 + \alpha \|\textbf{b}^{n+1}\|^2- \alpha \|\textbf{b}^{n}\|^2 +|q^{n+1}|^2-|q^n|^2  \\
& +|q^{n+1}-q^n|^2+\| \textbf{u} ^{n+1}-\textbf{u}^{n}\|^2 +
+\| \textbf{b}^{n+1}-\textbf{b}^{n}\|^2 \\
\hskip 1cm & \leq -2 \nu\Delta t \| \nabla \textbf{u}^{n+1} \|^2 - 2 \eta \alpha \Delta t \| \nabla \textbf{b}^{n+1} \|^2-\frac{2}{T}\Delta t|q^{n+1}|^2,
\endaligned
\end{equation*}
which implies the desired result.   
\end{proof}

We observe that the discrete energy dissipation law \eqref{discteng} is an approximation of the continuous energy dissipation law \eqref{engdiss2}.
\medskip

 \begin{theorem}\label{thm_energy stability_second order}
The scheme \eqref{e_SAV_scheme_second_u}-\eqref{e_SAV_scheme_second_q} is unconditionally stable in the sense that 
\begin{equation}\label{e_energy decay_second}
\aligned
E^{n+1}-E^{n}\leq -\Delta t (\nu\| \nabla \textbf{u}^{n+1} \|^2 +  \eta \alpha  \| \nabla \textbf{b}^{n+1} \|^2+\frac 1T|q^{n+1}|^2), \ \ \forall \Delta t,\; n\geq 0,
\endaligned
\end{equation} 
where 
\begin{equation}\label{e_definition of E}
\aligned
E^{n+1}= & \frac 14 (\| \textbf{u}^{n+1}\|^2 + \alpha \| \textbf{b}^{n+1}\|^2+|q^{n+1}|^2)\\
&+\frac 14 (\| 2\textbf{u}^{n+1}-\textbf{u}^{n} \|^2 + \alpha \| 2\textbf{b}^{n+1}-\textbf{b}^{n} \|^2  +|2q^{n+1}-q^n|^2) .
\endaligned
\end{equation} 
\end{theorem}

\begin{proof} 
Taking the inner product of \eqref{e_SAV_scheme_second_u} with   $ 4 \Delta t \textbf{u}^{n+1}$ and using the identity 
\begin{equation}\label{e_identity_BDF}
\aligned
2(3a-4b+c,a)=|a|^2+|2a-b|^2-|b|^2-|2b-c|^2+|a-2b+c|^2,
\endaligned
\end{equation} 
we have
\begin{equation}\label{e_stability_second_u}
\aligned
& \| \textbf{u} ^{n+1} \|^2 +  \| 2 \textbf{u} ^{n+1}- \textbf{u} ^{n} \|^2 - \| \textbf{u} ^{n} \|^2 -  \| 2 \textbf{u} ^{n}- \textbf{u} ^{n-1} \|^2 + \| \textbf{u} ^{n+1}-2\textbf{u} ^n+ \textbf{u} ^{n-1} \|^2 \\
&\ \ \ \ \ \ 
+ 4 \nu\Delta t \| \nabla \textbf{u}^{n+1} \|^2 + 4 \Delta t(\nabla p^{n+1}, \textbf{u}^{n+1}) \\
& = 4 \Delta t\exp( \frac{t^{n+1} }{T} ) q^{n+1} \left(  \alpha  ( (\nabla \times \bar{ \textbf{b} }^{n+1} ) \times \bar{ \textbf{b} }^{n+1}, \textbf{u}^{n+1} ) -( \bar{ \textbf{u} }^{n+1} \cdot \nabla \bar{ \textbf{u} }^{n+1},\textbf{u}^{n+1} )\right) .
\endaligned
\end{equation} 
Taking the inner product of \eqref{e_SAV_scheme_second_b} with   $ 4 \alpha \Delta t \textbf{b}^{n+1}$ leads to
\begin{equation}\label{e_stability_second_b}
\aligned
& \alpha ( \| \textbf{b} ^{n+1} \|^2 +  \| 2 \textbf{b} ^{n+1}- \textbf{b} ^{n} \|^2 - \| \textbf{b} ^{n} \|^2 -  \| 2 \textbf{b} ^{n}- \textbf{b} ^{n-1} \|^2 + \| \textbf{b} ^{n+1}-2\textbf{b} ^n+ \textbf{b} ^{n-1} \|^2 ) \\
& + 4 \eta \alpha \Delta t \| \nabla \textbf{b}^{n+1} \|^2 + 4 \alpha \Delta t  \exp( \frac{t^{n+1}}{T} ) q^{n+1} \left(  \nabla \times ( \bar{ \textbf{b} }^{n+1} \times \bar{ \textbf{u} }^{n+1} ) , \textbf{b}^{n+1} \right) =0. 
\endaligned
\end{equation} 
Multiplying \eqref{e_SAV_scheme_second_q} by $4 \Delta t q^{n+1}$ leads to
\begin{equation}\label{e_stability_second_q}
\aligned
& |q^{n+1}|^2+|2q^{n+1}-q^n|^2-|q^n|^2-|2q^{n}-q^{n-1}|^2+|q^{n+1}-2q^n+q^{n-1}|^2 \\
= & - \frac{4\Delta t}{T} |q^{n+1}|^2 + 4 \Delta t q^{n+1} \exp( \frac{t^{n+1} }{T} )( ( \bar{ \textbf{u} }^{n+1} \cdot \nabla ) \bar{ \textbf{u} }^{n+1}, \textbf{u}^{n+1} ) \\
& - 4 \alpha \Delta t q^{n+1} \exp( \frac{t^{n+1} }{T} ) \left(  ( (\nabla \times \bar{ \textbf{b} }^{n+1})  \times \bar{ \textbf{b} }^{n+1}, \textbf{u}^{n+1} ) -  (\nabla \times (\bar{ \textbf{b} }^{n+1} \times \bar{ \textbf{u} }^{n+1} ) , \textbf{b}^{n+1}) \right) . 
\endaligned
\end{equation}

Then summing up \eqref{e_stability_second_u} with \eqref{e_stability_second_b}-\eqref{e_stability_second_q} results in 
\begin{equation*}\label{e_stability_second_final}
\aligned
&  \| \textbf{u}^{n+1}\|^2+\| 2\textbf{u}^{n+1}-\textbf{u}^{n} \|^2 + \alpha \| \textbf{b}^{n+1}\|^2+ \alpha \| 2\textbf{b}^{n+1}-\textbf{b}^{n} \|^2 \\
&+ |q^{n+1}|^2+|2q^{n+1}-q^n|^2 + \| \textbf{u} ^{n+1}-2\textbf{u} ^n+ \textbf{u} ^{n-1} \|^2
+ \alpha \| \textbf{b} ^{n+1}-2\textbf{b} ^n+ \textbf{b} ^{n-1} \|^2 \\
& + |q^{n+1}-2q^n+q^{n-1}|^2 + \frac{4\Delta t}{T} |q^{n+1}|^2 + 4 \nu\Delta t \| \nabla \textbf{u}^{n+1} \|^2 + 4 \eta \alpha \Delta t \| \nabla \textbf{b}^{n+1} \|^2 \\
\leq &  \| \textbf{u}^{n}\|^2+\| 2\textbf{u}^{n}-\textbf{u}^{n-1} \|^2 + \alpha \| \textbf{b}^{n}\|^2+ \alpha \| 2\textbf{b}^{n}-\textbf{b}^{n-1} \|^2  
+ |q^{n}|^2+|2q^{n}-q^{n-1} |^2,
\endaligned
\end{equation*}
which implies the desired result.  
\end{proof}

Note that  the discrete energy  defined in  \eqref{e_definition of E} is  a second-order approximation of the continuous energy defined in \eqref{engdiss2}, and  \eqref{e_energy decay_second} is  an  approximation of the continuous energy dissipation law \eqref{engdiss2}.
  \section{Error Analysis} 
In this section, we carry out a rigorous error analysis for Scheme I \eqref{e_SAV_scheme_first_u}-\eqref{e_SAV_scheme_first_q} in the two-dimensional case. Similar analysis can also be carried out for Scheme II but the process  is much more tedious so we opt to only consider  Scheme I here. We emphasize that while both schemes can be used in the three-dimension case, the error analysis can not be easily extended to the three-dimension case due to some technical issues. Hence, we set $d=2$ in this section.

  \subsection{Preliminaries}
We describe below some notations and results which will be frequently used in the analysis.
We use $C$, with or without subscript, to denote a positive
constant, which could have different values at different places.

We  use the standard notations $L^2(\Omega)$, $H^k(\Omega)$ and $H^k_0(\Omega)$ to denote the usual Sobolev spaces. The norm corresponding to $H^k(\Omega)$ will be denoted simply by $\|\cdot\|_k$.  The vector functions and vector spaces will be indicated by boldface type.
  
  We define    
\begin{flalign*} 
    \begin{array}{l}
    \displaystyle  L^2_0( \Omega)  = 
\{ p \in   L^2 ( \Omega )  :  \int_{\Omega} q dx=0 \} , \\
\displaystyle  \textbf{H}^k( \Omega ) = ( H^k( \Omega) )^d,\ \  \textbf{H}^1_0( \Omega )  =
\{ \textbf{v} \in  \textbf{H}^1( \Omega ) :  \textbf{v} |_{\partial \Omega }=0 \}, \\
\displaystyle  \textbf{H}^1_n( \Omega )  =
\{ \textbf{v} \in  \textbf{H}^1( \Omega ) :  \textbf{v} \cdot \textbf{n}| _{\partial \Omega }= 0 \} , \\
\displaystyle  \textbf{V}  =
\{ \textbf{v} \in  \textbf{H}_0^1( \Omega ) : \nabla\cdot \textbf{v} =0 \} , \\
\displaystyle  \textbf{H}  =
\{ \textbf{v} \in  ( L^2 ( \Omega ) )^2 :  \nabla\cdot  \textbf{v} =0, \ \textbf{v} \cdot \textbf{n}| _{\partial \Omega }= 0 \} .
   \end{array}
  \end{flalign*}
The following formulae are essential and useful for our analysis 
\begin{equation}\label{e_cross_product1}
\aligned
( \nabla \times \textbf{v} ) \times \textbf{v} = ( \textbf{v} \cdot  \nabla ) \textbf{v} - \frac{1}{2} \nabla | \textbf{v}|^2,
\endaligned
\end{equation}
\begin{equation}\label{e_cross_product2}
\aligned
 \textbf{v} \times ( \textbf{w} \times \textbf{z} ) = (  \textbf{v} \cdot \textbf{z} ) \textbf{w} - 
(  \textbf{v} \cdot \textbf{w} ) \textbf{z}, 
\endaligned
\end{equation}
\begin{equation}\label{e_cross_product2_plus}
\aligned
 \nabla \times ( \textbf{v} \times \textbf{w} ) = (  \textbf{w} \cdot \nabla ) \textbf{v} - 
(  \textbf{v} \cdot \nabla ) \textbf{w} + ( \nabla \cdot \textbf{w} ) \textbf{v} - ( \nabla \cdot \textbf{v} ) \textbf{w}, 
\endaligned
\end{equation}
\begin{equation}\label{e_cross_product3}
\aligned
( \textbf{v} \times \textbf{w} ) \times \textbf{z} \cdot \textbf{q} = ( \textbf{v} \times \textbf{w} ) \cdot  ( \textbf{z} \times \textbf{q} ) = -  ( \textbf{v} \times \textbf{w} ) \cdot  ( \textbf{q} \times \textbf{z} ) , 
\endaligned
\end{equation}
\begin{equation}\label{e_integration by parts1}
\aligned
\int_{ \Omega } ( \nabla \times \textbf{v} ) \cdot \textbf{w} d \textbf{x} =  \int_{ \Omega } \textbf{v} \cdot ( \nabla \times \textbf{w} ) d \textbf{x} + \int_{ \partial \Omega } ( \textbf{n} \times \textbf{v} ) \cdot \textbf{w} ds.
\endaligned 
\end{equation}
Define the Stokes operator
  $$ A\textbf{u}=-P \Delta\textbf{u},\ \ \forall \ \textbf{u}\in D(A)=\textbf{H}^2(\Omega)\cap\textbf{V},$$
where $P $ is the orthogonal projector in $\textbf{L}^2(\Omega)$ onto $\textbf{H}$, and the Stokes operator $A$ is an unbounded positive self-adjoint closed operator in $\textbf{H}$ with domain $D(A)$. We then derive from the above and  Poincar\'e inequality that \cite{temam2001navier,heywood1982finite}
\begin{equation}\label{e_norm H2}
\aligned
\|\nabla\textbf{v}\|\leq c_1\|A^{\frac{1}{2}}\textbf{v}\|,\ \ \|\Delta\textbf{v}\|\leq c_1\|A\textbf{v}\|, \ \ \forall \ \textbf{v}\in D(A)=\textbf{H}^2(\Omega)\cap\textbf{V},
\endaligned
\end{equation} 
and 
\begin{equation}\label{e_norm H1}
\aligned
\|\textbf{v}\|\leq c_1\|\nabla\textbf{v}\|, \ \forall \ \textbf{v}\in \textbf{H}^1_0(\Omega),\ \ 
\|\nabla\textbf{v}\|\leq c_1\|A\textbf{v}\|, \ \ \forall \ \textbf{v}\in D(A) .
\endaligned
\end{equation}

 We recall the following inequalities  will be used in the sequel \cite{Gerbeau2006Mathematical,Jinjin2019A}:
\begin{equation}\label{e_norm curl}
\aligned
\|\nabla \times \textbf{v}\|_0 \leq c_1 \| \nabla \textbf{v} \|_0 , \ \ \|\nabla \cdot \textbf{v}\|_0 \leq c_1 \| \nabla \textbf{v} \|_0, \ \forall \  \textbf{v} \in \textbf{H}^1 (\Omega),
\endaligned
\end{equation}
\begin{equation}\label{e_norm curl_div}
\aligned
\|\nabla \times \textbf{v}\|_0^2 + \| \nabla \cdot \textbf{v}\|_0 ^2 \geq c_1 \| \textbf{v} \|_1^2,  \ \forall \  \textbf{v} \in \textbf{H}^1_n (\Omega),
\endaligned
\end{equation}
 and the following well-known inequalities which are valid with $d=2$  \cite{liu2010stable}: 
\begin{equation}\label{e_norm L4}
\aligned
\|\textbf{v}\|_{L^4} \leq c_1 \| \textbf{v}\|^{1/2}_0 \| \textbf{v}\|^{1/2}_1, \ \forall \  \textbf{v} \in \textbf{H}^1 (\Omega),
\endaligned
\end{equation}
\begin{equation}\label{e_norm L_infty}
\aligned
\|\textbf{v}\|_{L^{\infty}} \leq c_1 \| \textbf{v}\|^{1/2}_1 \| \textbf{v}\|^{1/2}_2, \ \forall \  \textbf{v} \in \textbf{H}^2 (\Omega),
\endaligned
\end{equation}
where  $c_1$ is a positive constant depending only on $\Omega$.

Next we define the trilinear form $b(\cdot,\cdot,\cdot)$ by
\begin{equation*}
\aligned
b(\textbf{u},\textbf{v},\textbf{w})=\int_{\Omega}(\textbf{u}\cdot\nabla)\textbf{v}\cdot \textbf{w}d\textbf{x}.
\endaligned
\end{equation*}
We can easily obtain that the trilinear form $b(\cdot,\cdot,\cdot)$ is a skew-symmetric with respect to its last two arguments, i.e., 
\begin{equation}\label{e_skew-symmetric1}
\aligned
b(\textbf{u},\textbf{v},\textbf{w})=-b(\textbf{u},\textbf{w},\textbf{v}),\ \ \forall \ \textbf{u}\in \textbf{H}, \ \ \textbf{v}, \textbf{w}\in \textbf{H}^1(\Omega),
\endaligned
\end{equation}
and 
\begin{equation}\label{e_skew-symmetric2}
\aligned
b(\textbf{u},\textbf{v},\textbf{v})=0,\ \ \forall \ \textbf{u}\in \textbf{H}, \ \ \textbf{v}\in \textbf{H}^1 (\Omega).
\endaligned
\end{equation}
By using a combination of integration by parts, Holder's inequality, and Sobolev inequalities\cite{Temam1995Navier,Shen1992On,he2013euler}, we have that for $d \leq 4$,
\begin{flalign}\label{e_estimate for trilinear form}
b(\textbf{u},\textbf{v},\textbf{w})\leq \left\{
   \begin{array}{l}
   c_2\|\textbf{u}\|_1\|\textbf{v}\|_1\|\textbf{w}\|_1,\\
   c_2\|\textbf{u}\|_2\|\textbf{v}\|\|\textbf{w}\|_1,\\
   c_2\|\textbf{u}\|_2\|\textbf{v}\|_1\|\textbf{w}\|,\\
   c_2\|\textbf{u}\|_1\|\textbf{v}\|_2\|\textbf{w}\|,\\
   c_2\|\textbf{u}\|\|\textbf{v}\|_2\|\textbf{w}\|_1,
   \end{array}
   \right.
\end{flalign}
and that for $d=2$, we have
\begin{flalign}\label{e_estimate for trilinear form1}
b(\textbf{u},\textbf{v},\textbf{w})\leq \left\{
   \begin{array}{l}
   c_2\|\textbf{u}\|_1^{1/2}\|\textbf{u}\|^{1/2}\|\textbf{v}\|_1^{1/2}\|\textbf{v}\|^{1/2}\|\textbf{w}\|_1, \\
   c_2\|\textbf{u}\|_1^{1/2}\|\textbf{u}\|^{1/2}\|A\textbf{v}\|^{1/2}\|\textbf{v}\|^{1/2}\|\textbf{w}\|, \\
   c_2\|A\textbf{u}\|^{1/2}\|\textbf{u}\|^{1/2}\|\textbf{v}\|_1\|\textbf{w}\|,
      \end{array}
   \right.
\end{flalign}
where $c_2$ is a positive constant depending only on $\Omega$. 

We will frequently use the following discrete version of the Gronwall lemma:

\medskip
\begin{lemma} \label{lem: gronwall2}
Let $a_k$, $b_k$, $c_k$, $d_k$, $\gamma_k$, $\Delta t_k$ be nonnegative real numbers such that
\begin{equation}\label{e_Gronwall3}
\aligned
a_{k+1}-a_k+b_{k+1}\Delta t_{k+1}+c_{k+1}\Delta t_{k+1}-c_k\Delta t_k\leq a_kd_k\Delta t_k+\gamma_{k+1}\Delta t_{k+1}
\endaligned
\end{equation}
for all $0\leq k\leq m$. Then
 \begin{equation}\label{e_Gronwall4}
\aligned
a_{m+1}+\sum_{k=0}^{m+1}b_k\Delta t_k \leq \exp \left(\sum_{k=0}^md_k\Delta t_k \right)\{a_0+(b_0+c_0)\Delta t_0+\sum_{k=1}^{m+1}\gamma_k\Delta t_k \}.
\endaligned
\end{equation}
\end{lemma}

Finally, we may drop the dependence on ${\bm x}$ if no confusion can arise. In particular, we  set
   \begin{numcases}{}
\displaystyle e_{\textbf{b}}^{n+1}=\textbf{b}^{n+1}-\textbf{b}(t^{n+1}),\ \ 
\displaystyle e_{\textbf{u}}^{n+1}=\textbf{u}^{n+1}-\textbf{u}(t^{n+1}), \notag\\
\displaystyle e_{p}^{n+1}=p^{n+1}-p(t^{n+1}),\ \ \ 
\displaystyle e_{q}^{n+1}=q^{n+1}-q(t^{n+1}).\notag
\end{numcases}

\subsection{Error estimates for the velocity and magnetic field} 
In this subsection, we  derive the following error estimates for the velocity $\textbf{u}$ and magnetic field $\textbf{b}$.

\begin{theorem}\label{thm: error_estimate_ubq}
Assuming $\textbf{u}\in H^2(0,T;\textbf{H}^{-1}(\Omega))\bigcap H^1(0,T;\textbf{H}^2(\Omega))\bigcap L^{\infty}(0,T; \textbf{H}^2(\Omega) )$,   and $\textbf{b}\in H^2(0,T;\textbf{H}^{-1}(\Omega))\bigcap H^1(0,T;\textbf{H}^2(\Omega))\bigcap L^{\infty}(0,T; \textbf{H}^2(\Omega) )$, 
then for the   scheme \eqref{e_SAV_scheme_first_u}-\eqref{e_SAV_scheme_first_q}, we have
\begin{equation*}
\aligned
& \| e_{\textbf{u}}^{m +1}\|^2 +  \| e_{\textbf{b}}^{m+1} \|^2 + |e_q^{m+1}|^2 +
\nu \Delta t \sum\limits_{n=0}^{m} \| \nabla e_{\textbf{u}}^{n+1}\|^2 \\
& + \eta \Delta t \sum\limits_{n=0}^{m} \| \nabla e_{\textbf{b}}^{n+1}\|^2 + \Delta t \sum\limits_{n=0}^{m} |  e_{q}^{n+1} |^2
 + \sum\limits_{n=0}^{m} \| e_{\textbf{u}}^{n+1}-e_{\textbf{u}}^n\|^2 \\
& + \sum\limits_{n=0}^{m} \| e_{\textbf{b}}^{n+1}-e_{\textbf{b}}^n\|^2 + 
\sum\limits_{n=0}^{m} | e_q^{n+1}-e_q^n |^2 \leq   C (\Delta t)^2  ,   \ \ \ \forall \ 0\leq n \leq N-1,
\endaligned
\end{equation*}
where $C$ is a positive constant  independent of $\Delta t$.
\end{theorem}

The proof of the above theorem will be carried out with a sequence of lemmas below. 

We  start first with the following  uniform bounds which are direct consequence of  the energy stability in Theorem \ref{thm_energy stability_first order}.
\begin{lemma}\label{lem_L2H1_boundedness}
Let ($\textbf{u}^{n+1}$, $p^{n+1}$, $q^{n+1}$, $\textbf{b}^{n+1}$) be the solution of  \eqref{e_SAV_scheme_first_u}-\eqref{e_SAV_scheme_first_q}, then we have
\begin{equation}\label{e_ubq_boundedness_L2 }
\aligned
\| \textbf {u}^{m+1} \|^2 + \| \textbf {b}^{m+1} \|^2 + | q^{m+1} |^2 \leq k_1, \ \ \forall \ 0\leq m\leq N-1,
\endaligned
\end{equation} 
and 
\begin{equation}\label{e_ub_boundedness_H1 }
\aligned
 \Delta t\sum_{n=0}^{m} \|  \textbf {u}^{n+1} \|_1^2 + \Delta t\sum_{n=0}^{m} \|  \textbf {b}^{n+1} \|_1^2  \leq k_2, \ \ \forall \ 0\leq m\leq N-1,
\endaligned
\end{equation} 
where the constants $k_i$ $(i=1,2)$ are independent of $\Delta t$.
\end{lemma}

Next, we derive a first bound for  the velocity  errors. 
\begin{lemma}\label{lem: error_estimate_u} 
Under the assumptions of Theorem \ref{thm: error_estimate_ubq}, we have
\begin{equation}\label{lem3.4}
\aligned
\frac{\| e_{\textbf{u}}^{n+1}\|^2-\|e_{\textbf{u}}^n\|^2}{2\Delta t}&+\frac{\| e_{\textbf{u}}^{n+1}-e_{\textbf{u}}^n\|^2}{2\Delta t}+ \frac{ \nu } {2} \| \nabla e_{\textbf{u}}^{n+1}\|^2  \\
\leq &  \exp( \frac{t^{n+1}}{T} ) e_q^{n+1}  \left(\alpha( (\nabla \times \textbf{b}^{n} ) \times \textbf{b}^{n}, e_{\textbf{u}}^{n+1})  - (\textbf{u}^n \cdot \nabla\textbf{u}^n, e_{\textbf{u}}^{n+1})\right) \\
&  + C ( \| \textbf{u}(t^n)\|_2^2 + \| \textbf{u}( t^{n+1} )\|_2^2 +\|e_{\textbf{u}}^n\|^2_1)  \|e_{\textbf{u}}^n\|^2   + C( \| e_{ \textbf{b} }^{n} \|_1^2 + \| \textbf{b}( t^{n+1} )  \|_2^2 ) \| e_{ \textbf{b} }^{n} \|^2   \\
 &+  C \Delta t \int_{t^n}^{t^{n+1}} (  \|\textbf{u}_t \|_2^2 + \|\textbf{u}_{tt}\|_{-1}^2 +  \|\textbf{b}_t \|^2_2 ) dt,   \ \ \ \forall \ 0\leq n \leq N-1,
\endaligned
\end{equation}
where $C$ is a positive constant  independent of $\Delta t$.
\end{lemma}

\begin{proof}
 Let $\textbf{R}_{\textbf{u}}^{n+1}$ be the truncation error defined by 
\begin{equation}\label{e_error_Ru}
\aligned
\textbf{R}_{\textbf{u}}^{n+1}=\frac{\partial \textbf{u}(t^{n+1})}{\partial t}- \frac{\textbf{u}(t^{n+1})-\textbf{u}(t^{n})}{\Delta t}=\frac{1}{\Delta t}\int_{t^n}^{t^{n+1}}(t^n-t)\frac{\partial^2 \textbf{u}}{\partial t^2}dt.
\endaligned
\end{equation}
Subtracting \eqref{e_MHD_model_SAV1} at $t^{n+1}$ from \eqref{e_SAV_scheme_first_u}, we obtain
\begin{equation}\label{e_error_u}
\aligned
 d_t e_{\textbf{u} }^{n+1} &- \nu \Delta e_{\textbf{u} }^{n+1} + \nabla e_p^{n+1} 
= \textbf{R}_{\textbf{u}}^{n+1} \\
&+ \exp( \frac{t^{n+1}}{T} ) q(t^{n+1}) (\textbf{u}(t^{n+1})\cdot \nabla\textbf{u}(t^{n+1}) 
- \textbf{u}^{n}\cdot \nabla  \textbf{u}^{n})\\
&+ \alpha \exp( \frac{t^{n+1}}{T} ) q^{n+1} ( (\nabla \times \textbf{b}^{n})  \times \textbf{b}^{n} - (\nabla \times \textbf{b}(t^{n+1}))  \times \textbf{b}(t^{n+1} )).
\endaligned
\end{equation}
Taking the inner product of \eqref{e_error_u} with $e_{\textbf{u}}^{n+1}$, we obtain
\begin{equation}\label{e_error_u_inner_product}
\aligned
&\frac{\| e_{\textbf{u}}^{n+1}\|^2-\|e_{\textbf{u}}^n\|^2}{2\Delta t}+\frac{\| e_{\textbf{u}}^{n+1}-e_{\textbf{u}}^n\|^2}{2\Delta t}+\nu \| \nabla e_{\textbf{u}}^{n+1}\|^2 + ( \nabla e_p^{n+1}, e_{\textbf{u}}^{n+1} ) =(\textbf{R}_{\textbf{u}}^{n+1}, e_{\textbf{u}}^{n+1})\\
& + \exp( \frac{t^{n+1}}{T} ) \left( q(t^{n+1}) \textbf{u}(t^{n+1})\cdot \nabla\textbf{u}(t^{n+1})-q^{n+1}  \textbf{u}^{n}\cdot \nabla  \textbf{u}^{n}, e_{\textbf{u}}^{n+1} \right)   \\
& + \alpha \exp( \frac{t^{n+1}}{T} ) \left( q^{n+1}  ( \nabla \times \textbf{b}^{n} ) \times \textbf{b}^{n} - q(t^{n+1}) ( \nabla \times \textbf{b}(t^{n+1}) ) \times \textbf{b}(t^{n+1} )
, e_{\textbf{u}}^{n+1} \right) .
\endaligned
\end{equation}
For the first term on the right hand side of \eqref{e_error_u_inner_product}, we have 
\begin{equation}\label{e_error_inner_Ru}
\aligned
&(\textbf{R}_{\textbf{u}}^{n+1}, e_{\textbf{u}}^{n+1})\leq \frac{\nu}{16} \|\nabla e_{\textbf{u}}^{n+1} \|^2+C \Delta t\int_{t^n}^{t^{n+1}}\|\textbf{u}_{tt}\|_{-1}^2dt.
\endaligned
\end{equation}
For the second term on the right hand side of \eqref{e_error_u_inner_product}, we have
\begin{equation}\label{e_error_u_nonlinear_convective}
\aligned
 \exp( \frac{t^{n+1}}{T} ) &\left( q(t^{n+1}) \textbf{u}(t^{n+1})\cdot \nabla\textbf{u}(t^{n+1})-q^{n+1} \textbf{u}^{n}\cdot \nabla\textbf{u}^{n}, e_{\textbf{u}}^{n+1} \right)  \\
=&  \left(  (  \textbf{u}(t^{n+1})-\textbf{u}^{n} )\cdot \nabla  \textbf{u}(t^{n+1}), e_{\textbf{u}}^{n+1}\right) +  \left(  \textbf{u}^{n}\cdot \nabla  (\textbf{u}(t^{n+1})-\textbf{u}^{n} ), e_{\textbf{u}}^{n+1}\right)\\
&-  \exp( \frac{t^{n+1}}{T} )  e_q^{n+1} \left(\textbf{u}^n \cdot \nabla\textbf{u}^n, e_{\textbf{u}}^{n+1}\right).
\endaligned
\end{equation}
Using Cauchy-Schwarz inequality and recalling  Lemma \ref{lem_L2H1_boundedness} and  \eqref{e_estimate for trilinear form},  the first term on the right hand side of \eqref{e_error_u_nonlinear_convective} can be bounded by
\begin{equation}\label{e_error_u_nonlinear_convective1}
\aligned
&  \left(   ( \textbf{u}(t^{n+1})-\textbf{u}^{n} )\cdot \nabla  \textbf{u}(t^{n+1}), e_{\textbf{u}}^{n+1}\right)\\
& \ \ \ \ \
\leq c_2(1+c_1) \| \textbf{u}(t^{n+1})-\textbf{u}^{n} \| \| \textbf{u}(t^{n+1})\|_2
\|\nabla e_{\textbf{u}}^{n+1} \| \\
& \ \ \ \ \
\leq \frac{\nu}{16} \|\nabla e_{\textbf{u}}^{n+1} \|^2+C \| \textbf{u}(t^{n+1})\|_2^2\|e_{\textbf{u}}^n\|^2+C \| \textbf{u}(t^{n+1})\|_2^2\Delta t\int_{t^n}^{t^{n+1}}\|\textbf{u}_t\|^2dt.
\endaligned
\end{equation}
The second term on the right hand side of \eqref{e_error_u_nonlinear_convective} can be estimated as follows by using the similar procedure in \cite{li2020new},
\begin{equation}\label{e_error_u_nonlinear_convective2}
\aligned
(   \textbf{u}^{n}\cdot \nabla & (\textbf{u}(t^{n+1})-\textbf{u}^{n} ), e_{\textbf{u}}^{n+1})\\
= & \left(  \textbf{u}^{n}\cdot \nabla  (\textbf{u}(t^{n+1})-\textbf{u}(t^{n}) ), e_{\textbf{u}}^{n+1}\right)-  \left(  e_{\textbf{u}}^{n}\cdot \nabla  e_{\textbf{u}}^n, e_{\textbf{u}}^{n+1}\right)-  \left(  \textbf{u}(t^n)\cdot \nabla  e_{\textbf{u}}^n, e_{\textbf{u}}^{n+1}\right)\\
\leq &c_2(1+c_1)  \|\nabla e_{\textbf{u}}^{n+1} \| (\|\textbf{u}^{n}\|  \| \int_{t^n}^{t^{n+1}}\textbf{u}_tdt \|_2+\|e_{\textbf{u}}^n\| \| \textbf{u}(t^n)\|_2)\\
&+c_2(1+c_1)  \|e_{\textbf{u}}^n\|^{1/2} \|e_{\textbf{u}}^n\|^{1/2}_1\|e_{\textbf{u}}^n\|^{1/2} \|e_{\textbf{u}}^n\|^{1/2}_1\| \nabla e_{\textbf{u}}^{n+1} \| \\
\leq & \frac{\nu}{16} \|\nabla e_{\textbf{u}}^{n+1} \|^2+C (\| \textbf{u}(t^n)\|_2^2+\|e_{\textbf{u}}^n\|^2_1)  \|e_{\textbf{u}}^n\|^2+C \Delta t \int_{t^n}^{t^{n+1}}\|\textbf{u}_t \|_2^2dt.
\endaligned
\end{equation}
For the last term on the right hand side of \eqref{e_error_u_inner_product}, we have
\begin{equation}\label{e_error_u_Lorentz_force}
\aligned
 \exp( \frac{t^{n+1}}{T} )& \left( q^{n+1}  ( \nabla \times \textbf{b}^{n} ) \times \textbf{b}^{n} - q(t^{n+1}) ( \nabla \times \textbf{b}(t^{n+1}) ) \times \textbf{b}(t^{n+1} )
, e_{\textbf{u}}^{n+1} \right) \\
=& \exp( \frac{t^{n+1}}{T} ) e_q^{n+1}  \left( ( \nabla \times \textbf{b}^{n} ) \times \textbf{b}^{n}, e_{\textbf{u}}^{n+1} \right) +  \left(  (\nabla \times (\textbf{b}^{n}- \textbf{b}(t^{n+1}) )) \times \textbf{b}^{n}, e_{\textbf{u}}^{n+1} \right) \\
& +  \left(  ( \nabla \times \textbf{b} (t^{n+1}) ) \times ( \textbf{b}^{n} - \textbf{b}(t^{n+1}) ), e_{\textbf{u}}^{n+1} \right) .
\endaligned
\end{equation}
The second term on the right hand side of \eqref{e_error_u_Lorentz_force} can be transformed into 
\begin{equation}\label{e_error_u_Lorentz_force1}
\aligned
 (  (\nabla \times (\textbf{b}^{n}- \textbf{b}(t^{n+1}) )) &\times \textbf{b}^{n}, e_{\textbf{u}}^{n+1} ) \\
=& \left( ( \nabla \times e_{ \textbf{b} }^{n}) \times e_ { \textbf{b} }^{n}, e_{\textbf{u}}^{n+1} \right)  +  \left( ( \nabla \times e_{ \textbf{b} }^{n} ) \times   \textbf{b} (t^n), e_{\textbf{u}}^{n+1} \right) \\
& + \left(  (\nabla \times ( \textbf{b}(t^n)- \textbf{b}(t^{n+1}) ) ) \times \textbf{b}^{n}, e_{\textbf{u}}^{n+1} \right).
\endaligned
\end{equation}
Using the identity \eqref{e_cross_product1}, the first term on the right hand side of \eqref{e_error_u_Lorentz_force1} can be bounded by
\begin{equation}\label{e_error_u_Lorentz_force2}
\aligned
& \left(  (\nabla \times e_{ \textbf{b} }^{n}) \times e_ { \textbf{b} }^{n}, e_{\textbf{u}}^{n+1} \right) 
=    \left( ( e_{ \textbf{b} }^{n} \cdot \nabla ) e_{ \textbf{b} }^{n}, e_{\textbf{u}}^{n+1} \right) - \frac{1}{2} \left(  \nabla |e_{ \textbf{b} }^{n}|^2 , e_{\textbf{u}}^{n+1} \right) \\
& \ \ \ \ \ 
\leq  C \| e_{ \textbf{b} }^{n} \|^{1/2}  \| e_{ \textbf{b} }^{n} \|_1^{1/2}  \| e_{ \textbf{b} }^{n} \|^{1/2}  \| e_{ \textbf{b} }^{n} \|_1^{1/2} \| \nabla e_{\textbf{u}}^{n+1} \| \\
& \ \ \ \ \ 
\leq   \frac{ \nu }{16} \| \nabla e_{\textbf{u}}^{n+1} \|^2 + C \| e_{ \textbf{b} }^{n} \|_1^2 \| e_{ \textbf{b} }^{n} \|^2  .
\endaligned
\end{equation} 
Using \eqref{e_cross_product2}, \eqref{e_cross_product3} and integration by parts \eqref{e_integration by parts1}, the second term on the right hand side of \eqref{e_error_u_Lorentz_force1} can be controlled by
\begin{equation}\label{e_error_u_Lorentz_force3}
\aligned
 \left(  (\nabla \times e_{ \textbf{b} }^{n} )\times   \textbf{b} (t^n), e_{\textbf{u}}^{n+1} \right) =&  -   \left(  e_{\textbf{u}}^{n+1} \times  \textbf{b} (t^n), \nabla \times e_{ \textbf{b} }^{n} \right)  \\
=& -  \left(  \nabla \times (e_{\textbf{u}}^{n+1} \times  \textbf{b} (t^n) ), e_{ \textbf{b} }^{n} \right) - < \textbf{n} \times (e_{\textbf{u}}^{n+1} \times  \textbf{b} (t^n) ), e_{ \textbf{b} }^{n}
> \\ 
=& \left(  ( e_{\textbf{u}}^{n+1} \cdot \nabla  ) \textbf{b} (t^n) , e_{ \textbf{b} }^{n} \right)
 -  \left(  (  \textbf{b} (t^n) \cdot   \nabla ) e_{\textbf{u}}^{n+1}, e_{ \textbf{b} }^{n} \right) \\
 \leq &  \frac{ \nu }{16} \| \nabla e_{\textbf{u}}^{n+1} \|^2 + C  \| \textbf{b}(t^n) \|_2^2  \| e_{ \textbf{b} }^{n} \| ^2, 
\endaligned
\end{equation}
where we use the identity 
\begin{equation*}
\aligned
 \nabla \times ( \textbf{v} \times \textbf{w} ) = (  \textbf{w} \cdot \nabla ) \textbf{v} - 
(  \textbf{v} \cdot \nabla ) \textbf{w}, \  \forall \ \textbf{v}, \textbf{w} \in \textbf{H}.
\endaligned
\end{equation*}
Lemma \ref{lem_L2H1_boundedness} and  \eqref{e_estimate for trilinear form}, the last term on the right hand side of \eqref{e_error_u_Lorentz_force1} can be estimated by
\begin{equation}\label{e_error_u_Lorentz_force4}
\aligned
&  \left(  (\nabla \times ( \textbf{b}(t^n)- \textbf{b}(t^{n+1}) )) \times \textbf{b}^{n}, e_{\textbf{u}}^{n+1} \right) \\
& \ \ \ \ \ \ 
\leq  \frac{ \nu }{16} \| \nabla e_{\textbf{u}}^{n+1} \|^2  + C \| \textbf{b}^{n} \|^2 \Delta t \int_{t^n}^{t^{n+1}} \|\textbf{b}_t \|^2_2 dt .
\endaligned
\end{equation}
For the last term on the right hand side of \eqref{e_error_u_Lorentz_force}, we have
\begin{equation}\label{e_error_u_Lorentz_force5}
\aligned
& \left(  ( \nabla \times \textbf{b} (t^{n+1}) ) \times ( \textbf{b}^{n} - \textbf{b}(t^{n+1}) ), e_{\textbf{u}}^{n+1} \right)  \\
& \ \ \ \ \ \ 
\leq \frac{ \nu }{16} \| \nabla e_{ \textbf{u} }^{n+1} \|^2 + C \| \textbf{b}( t^{n+1} )  \|_2^2 \| e_{ \textbf{b} } ^n \|^2 + C \Delta t \int_{t^n}^{t^{n+1}} \| \textbf{b}_t \|^2 dt . 
\endaligned
\end{equation}
Finally,  combining \eqref{e_error_u_inner_product} with \eqref{e_error_u_nonlinear_convective}-\eqref{e_error_u_Lorentz_force5} leads to the desired result.
\end{proof}

\medskip
We derive below a bound for the errors of the magnetic field.
\begin{lemma}\label{lem: error_estimate_b} Under the assumptions of Theorem \ref{thm: error_estimate_ubq},
 we have
\begin{equation}\label{lem3.5}
\aligned
\frac{\| e_{\textbf{b}}^{n+1}\|^2-\|e_{\textbf{b}}^n\|^2}{2\Delta t}&+\frac{\| e_{\textbf{b}}^{n+1}-e_{\textbf{b}}^n\|^2}{2\Delta t} +  \frac{\eta}{ 2 } \| \nabla e_{\textbf{b}}^{n+1}\|^2  \\
\le &  -  \exp( \frac{t^{n+1}}{T} ) e_q^{n+1} \left( \nabla \times ( \textbf{b}^n \times \textbf{u}^n ), e_{\textbf{b}}^{n+1} \right) + C ( \| \textbf{u}(t^{n+1}) \|_2^2 + \| e_{ \textbf{b} }^n \|_1^2 )  \| e_{ \textbf{b} } ^n \|^2  \\
&  + C ( \| e_{ \textbf{u} }^n \|_1^2+  \|  \textbf{b}(t^{n+1}) \|^2_2 ) \| e_{ \textbf{u} }^n \|^2 
+C \Delta t \int_{t^n}^{t^{n+1}}  \|\textbf{u}_t \|^2_2  dt  \\
 & +  C  \Delta t \int_{t^n}^{t^{n+1}} ( \|\textbf{b}_t \|^2 + \|\textbf{b}_{tt}\|_{-1}^2 )dt  ,  \ \ \ \forall \ 0\leq n \leq N-1,
\endaligned
\end{equation}
where $C$ is a positive constant  independent of $\Delta t$.
\end{lemma}

\begin{proof}
 Let $\textbf{R}_{\textbf{b}}^{n+1}$ be the truncation error defined by 
\begin{equation}\label{e_error_Rb}
\aligned
\textbf{R}_{\textbf{b}}^{n+1}=\frac{\partial \textbf{b}(t^{n+1})}{\partial t}- \frac{\textbf{b}(t^{n+1})-\textbf{b}(t^{n})}{\Delta t}=\frac{1}{\Delta t}\int_{t^n}^{t^{n+1}}(t^n-t)\frac{\partial^2 \textbf{b}}{\partial t^2}dt.
\endaligned
\end{equation}
Subtracting \eqref{e_MHD_model_SAV2} at $t^{n+1}$ from \eqref{e_SAV_scheme_first_b} and using \eqref{e_curl(curl b)}, we obtain
\begin{equation}\label{e_error_b}
\aligned
 d_t e_{\textbf{b} }^{n+1} - \eta \Delta e_{\textbf{b} }^{n+1} 
= & \exp( \frac{t^{n+1}}{T} ) q(t^{n+1}) \nabla \times ( \textbf{b}(t^{n+1}) \times \textbf{u}(t^{n+1}) ) \\
& - \exp( \frac{t^{n+1} }{T} ) q^{n+1} \nabla \times ( \textbf{b}^n \times \textbf{u}^n )
+\textbf{R}_{\textbf{u}}^{n+1} . 
\endaligned
\end{equation}
Taking the inner product of \eqref{e_error_b} with $e_{\textbf{b}}^{n+1}$, we obtain
\begin{equation}\label{e_error_b_inner_product}
\aligned
\frac{\| e_{\textbf{b}}^{n+1}\|^2-\|e_{\textbf{b}}^n\|^2}{2\Delta t}&+\frac{\| e_{\textbf{b}}^{n+1}-e_{\textbf{b}}^n\|^2}{2\Delta t} + \eta \| \nabla e_{\textbf{b}}^{n+1}\|^2  \\
=&  \exp( \frac{t^{n+1}}{T} ) q(t^{n+1}) \left(  \nabla \times ( \textbf{b}(t^{n+1}) \times \textbf{u}(t^{n+1}) )  , e_{\textbf{b}}^{n+1} \right)  \\
&  - \exp( \frac{t^{n+1}}{T} ) q^{n+1} \left(   \nabla \times ( \textbf{b}^n \times \textbf{u}^n ) , e_{\textbf{b}}^{n+1} \right)
+(\textbf{R}_{\textbf{b}}^{n+1}, e_{\textbf{b}}^{n+1}) .\endaligned
\end{equation}
The first two terms on the right hand side of \eqref{e_error_b_inner_product} can be recast as
\begin{equation}\label{e_error_b_nonlinear_convective}
\aligned
 \exp( \frac{t^{n+1}}{T} )& \left(   q(t^{n+1}) \nabla \times ( \textbf{b}(t^{n+1}) \times \textbf{u}(t^{n+1}) ) - q^{n+1} \nabla \times ( \textbf{b}^n \times \textbf{u}^n ), e_{\textbf{b}}^{n+1} \right)  \\
= & \left(  \nabla \times [ ( \textbf{b}(t^{n+1}) - \textbf{b}^n ) \times \textbf{u}(t^{n+1}) ] , e_{\textbf{b}}^{n+1} \right) + \left(  \nabla \times [ \textbf{b}^n  \times ( \textbf{u}(t^{n+1}) - \textbf{u}^n ) ] , e_{\textbf{b}}^{n+1} \right)  \\
& -  \exp( \frac{t^{n+1}}{T} ) e_q^{n+1} \left( \nabla \times ( \textbf{b}^n \times \textbf{u}^n ), e_{\textbf{b}}^{n+1} \right) .
\endaligned
\end{equation}
By using \eqref{e_norm L_infty}, \eqref{e_norm curl} and integration by parts \eqref{e_integration by parts1}, 
we have 
\begin{equation}\label{e_error_b_nonlinear_convective1}
\aligned
  \big(  \nabla \times &[ ( \textbf{b}(t^{n+1}) - \textbf{b}^n ) \times \textbf{u}(t^{n+1}) ] , e_{\textbf{b}}^{n+1} \big) 
=  \left(  ( \textbf{b}(t^{n+1}) - \textbf{b}^n ) \times \textbf{u}(t^{n+1})  , \nabla \times e_{\textbf{b}}^{n+1} \right) \\
\leq & \frac{\eta}{ 6 } \| \nabla e_{\textbf{b}}^{n+1} \|^2 + C \| \textbf{u}(t^{n+1}) \|_2^2  e_{ \textbf{b} } ^n \|^2  
  +  C \| \textbf{u}(t^{n+1}) \|_2^2 \Delta t \int_{t^n}^{t^{n+1}} \|\textbf{b}_t \|^2 dt .
\endaligned
\end{equation}
Thanks to \eqref{e_norm L4} and \eqref{e_norm curl}, we have
\begin{equation}\label{e_error_b_nonlinear_convective2}
\aligned
 \big(  \nabla \times &[ \textbf{b}^n  \times ( \textbf{u}(t^{n+1}) - \textbf{u}^n ) ] , e_{\textbf{b}}^{n+1} \big)  \\
= & \left(  \textbf{b}^n  \times ( \textbf{u}(t^{n+1}) - \textbf{u}^n )  , \nabla \times e_{\textbf{b}}^{n+1} \right) \\
=&  \left(  e_{ \textbf{b} }^n  \times ( \textbf{u}(t^{n+1}) - \textbf{u}(t^n) )  , \nabla \times e_{\textbf{b}}^{n+1} \right) - \left(   e_{ \textbf{b} }^n  \times  e_{ \textbf{u} }^n , \nabla \times e_{\textbf{b}}^{n+1} \right) \\
& +  \left(  \textbf{b}(t^{n+1})  \times ( \textbf{u}(t^{n+1}) - \textbf{u}^n )  , \nabla \times e_{\textbf{b}}^{n+1} \right) \\
\leq &  \frac{\eta}{ 6 } \| \nabla e_{\textbf{b}}^{n+1} \|^2 + C \| e_{ \textbf{b} }^n \|^2_{L^4} \| e_{ \textbf{u} }^n \|^2_{L^4} + C  \|  \textbf{b}(t^{n+1}) \|^2_2 \| e_{ \textbf{u} }^n \|^2 \\
& + C \| e_{ \textbf{b} }^n \|^2 \Delta t \int_{t^n}^{t^{n+1}} \|\textbf{u}_t \|^2_2 dt +
C \|  \textbf{b}(t^{n+1}) \|^2_2 \Delta t \int_{t^n}^{t^{n+1}} \|\textbf{u}_t \|^2 dt \\
\leq &  \frac{\eta}{ 6 } \| \nabla e_{\textbf{b}}^{n+1} \|^2 + C \| e_{ \textbf{b} }^n \|_1^2 \| e_{ \textbf{b} }^n \|^2  + C ( \| e_{ \textbf{u} }^n \|_1^2+  \|  \textbf{b}(t^{n+1}) \|^2_2 ) \| e_{ \textbf{u} }^n \|^2  \\
& + C \| e_{ \textbf{b} }^n \|^2 \Delta t \int_{t^n}^{t^{n+1}} \|\textbf{u}_t \|^2_2 dt +
C \|  \textbf{b}(t^{n+1}) \|^2_2 \Delta t \int_{t^n}^{t^{n+1}} \|\textbf{u}_t \|^2 dt .
\endaligned
\end{equation}
For the last term on the right hand side of \eqref{e_error_b_inner_product}, we have 
\begin{equation}\label{e_error_inner_Rb}
\aligned
&(\textbf{R}_{\textbf{b}}^{n+1}, e_{\textbf{b}}^{n+1})\leq \frac{ \eta }{6} \|\nabla e_{\textbf{b}}^{n+1} \|^2+C \Delta t\int_{t^n}^{t^{n+1}}\|\textbf{b}_{tt}\|_{-1}^2dt.
\endaligned
\end{equation}
Combining \eqref{e_error_b_inner_product} with \eqref{e_error_b_nonlinear_convective}-\eqref{e_error_inner_Rb} leads to  the desired result.
\end{proof}

\medskip
In the next lemma, we derive a bound for the errors with respect to $q$.

\begin{lemma}\label{lem: error_estimate_q} Under the assumptions of Theorem \ref{thm: error_estimate_ubq},
 we have
\begin{equation}\label{lem3.6}
\aligned
 \frac{|e_q^{n+1}|^2-|e_q^n|^2}{2\Delta t}&+\frac{|e_q^{n+1}-e_q^n|^2}{2\Delta t}+\frac{1}{2T}|e_q^{n+1}|^2  \\
\leq &  \exp( \frac{t^{n+1}}{T} )  e_q^{n+1} \left(\textbf{u}^n \cdot \nabla\textbf{u}^n, e_{\textbf{u}}^{n+1}\right) -  \alpha \exp( \frac{t^{n+1}}{T} ) e_q^{n+1}  \left( ( \nabla \times \textbf{b}^{n} ) \times \textbf{b}^{n}, e_{\textbf{u}}^{n+1} \right) \\ 
& + \alpha \exp( \frac{t^{n+1}}{T} ) e_q^{n+1} \left( \nabla \times ( \textbf{b}^n \times \textbf{u}^n ), e_{\textbf{b}}^{n+1} \right)  +  C  \|\textbf{u}^n \|_1^2 \| e_{\textbf{u}}^n \|^2  \\
& + C  ( \| e_{\textbf{b}}^n \|_1^2 +  \| \textbf{u} ^n \|_1^2 + \| \textbf{b}( t^{n+1} ) \|_1^2 ) \| e_{\textbf{b}}^n \|^2  + C \Delta t\int_{t^n}^{t^{n+1} } \|q_{tt} \|^2dt  \\
& + C \Delta t \int_{t^n}^{t^{n+1} } ( \|\textbf{u}_t \|_0^2 + \|\textbf{b}_t \|^2_1
)dt  ,   \ \ \ \forall \ 0\leq n \leq N-1,
\endaligned
\end{equation}
where $C$ is a positive constant  independent of $\Delta t$.
\end{lemma}

\begin{proof}
Subtracting \eqref{e_MHD_model_SAV3} from \eqref{e_SAV_scheme_first_q} leads to
\begin{equation}\label{e_error_q}
\aligned
 \frac{e_q^{n+1}-e_q^n}{\Delta t}&+\frac{1}{T}e_q^{n+1} =  \textbf{R}_{q}^{n+1} \\ 
& + \exp( \frac{t^{n+1}}{T} ) ( (\textbf{u}^n\cdot\nabla  \textbf{u}^n, \textbf{u}^{n+1} ) - (\textbf{u}(t^{n+1})\cdot \nabla  \textbf{u}(t^{n+1}),\textbf{u}(t^{n+1}) )) \\
  & 
- \alpha   \exp( \frac{ t^{n+1} }{T} ) \left(  ( (\nabla \times \textbf{b}^n)  \times \textbf{b}^n, \textbf{u}^{n+1} ) - ( (\nabla \times \textbf{b}(t^{n+1})) \times \textbf{b}(t^{n+1}), \textbf{u}(t^{n+1}) ) \right) \\
  & 
  + \alpha   \exp( \frac{t^{n+1}}{T} ) \left(  ( \nabla \times ( \textbf{b}^n \times \textbf{u}^n ) , \textbf{b}^{n+1} ) - ( \nabla \times ( \textbf{b}(t^{n+1}) \times \textbf{u}(t^{n+1}) ) , \textbf{b}(t^{n+1}) ) \right),
\endaligned
\end{equation}
where 
\begin{equation}\label{e_error_Rq}
\aligned
\textbf{R}_{q}^{n+1}=\frac{ \rm{d} q(t^{n+1})}{ \rm{d} t}- \frac{q(t^{n+1})-q(t^{n})}{\Delta t}=\frac{1}{\Delta t}\int_{t^n}^{t^{n+1}}(t^n-t)\frac{\partial^2 q}{\partial t^2}dt.
\endaligned
\end{equation}
Multiplying both sides of \eqref{e_error_q} by $e_q^{n+1}$ yields
\begin{equation}\label{e_error_q_inner}
\aligned
 &\frac{|e_q^{n+1}|^2-|e_q^n|^2}{2\Delta t}+\frac{|e_q^{n+1}-e_q^n|^2}{2\Delta t}+\frac{1}{T}|e_q^{n+1}|^2 = \textbf{R}_{q}^{n+1}  e_q^{n+1}\\
& \ \ \ \ \  + \exp( \frac{t^{n+1}}{T} ) e_q^{n+1} ( ( \textbf{u}^n\cdot\nabla  \textbf{u}^n, \textbf{u}^{n+1} ) 
  -  (\textbf{u}(t^{n+1})\cdot \nabla  \textbf{u}(t^{n+1}),\textbf{u}(t^{n+1}) )) \\
  & \ \ \ \ \ 
- \alpha   \exp( \frac{ t^{n+1} }{T} ) e_q^{n+1} \left(  ( (\nabla \times \textbf{b}^n ) \times \textbf{b}^n, \textbf{u}^{n+1} ) - ( (\nabla \times \textbf{b}(t^{n+1}) )\times \textbf{b}(t^{n+1}), \textbf{u}(t^{n+1}) ) \right) \\
  & \ \ \ \ \ 
  + \alpha   \exp( \frac{t^{n+1}}{T} ) e_q^{n+1} \left(  ( \nabla \times ( \textbf{b}^n \times \textbf{u}^n ) , \textbf{b}^{n+1} ) - ( \nabla \times ( \textbf{b}(t^{n+1}) \times \textbf{u}(t^{n+1}) ) , \textbf{b}(t^{n+1}) ) \right) .
\endaligned
\end{equation}
We bound the right hand side of the above as follows:
\begin{equation}\label{e_error_q_nonlinear4}
\aligned
\textbf{R}_{q}^{n+1}e_q^{n+1} \leq \frac{1}{12 T}  |e_q^{n+1}|^2+C \Delta t\int_{t^n}^{t^{n+1} } \|q_{tt} \|^2dt .
\endaligned
\end{equation}
The second term  on the right hand side of \eqref{e_error_q_inner} can be estimated as
\begin{equation}\label{e_error_q_nonlinear1}
\aligned
 & \exp( \frac{t^{n+1}}{T} ) e_q^{n+1} \left( ( \textbf{u}^n\cdot\nabla  \textbf{u}^n, \textbf{u}^{n+1} ) -  (  \textbf{u}(t^{n+1})\cdot \nabla  \textbf{u}(t^{n+1}),\textbf{u}(t^{n+1}) )  \right) \\
 = & \exp( \frac{t^{n+1}}{T} )  e_q^{n+1} \left(\textbf{u}^n \cdot \nabla\textbf{u}^n, e_{\textbf{u}}^{n+1}\right)  +  \exp( \frac{t^{n+1}}{T} )  e_q^{n+1} \left(   \textbf{u}^n\cdot\nabla  ( \textbf{u}^n - \textbf{u}(t^{n+1}) ), \textbf{u}(t^{n+1}) \right) \\
 & +  \exp( \frac{t^{n+1}}{T} )  e_q^{n+1}  \left(  ( 
\textbf{u}^n -  \textbf{u}(t^{n+1})) \cdot \nabla  \textbf{u}(t^{n+1}), \textbf{u}(t^{n+1}) \right) .
\endaligned
\end{equation}
Thanks to \eqref{e_estimate for trilinear form} and Lemma \ref{lem_L2H1_boundedness}, we bound the second term  on the right hand side of \eqref{e_error_q_nonlinear1} by
\begin{equation}\label{e_error_q_nonlinear2}
\aligned
&  \exp( \frac{t^{n+1}}{T} )  e_q^{n+1} \left(   \textbf{u}^n\cdot\nabla  ( \textbf{u}^n - \textbf{u}(t^{n+1}) ), \textbf{u}(t^{n+1}) \right) \\ 
& \ \ \ \ \ 
 \leq  C \|\textbf{u}^n \|_1 \|  \textbf{u}(t^{n+1})-  \textbf{u}(t^{n}) - e_{ \textbf{u} }^n  \|_0 \| \textbf{u}(t^{n+1}) \|_{2} |e_q^{n+1}| \\
 & \ \ \ \ \ 
 \leq \frac{1}{12T} |e_q^{n+1}|^2 + C  \|\textbf{u}^n \|_1^2 \| e_{\textbf{u}}^n \|^2 
+C  \| \textbf{u}(t^{n+1}) \|_{2} ^2 \Delta t \int_{t^n}^{t^{n+1} } \|\textbf{u}_t \|_0^2dt .
 \endaligned
\end{equation}
The third term on the right hand side of \eqref{e_error_q_nonlinear1} can be bounded by
\begin{equation}\label{e_error_q_nonlinear3}
\aligned
& \exp( \frac{t^{n+1}}{T} )  e_q^{n+1}  \left( ( 
\textbf{u}^n -  \textbf{u}(t^{n+1}) )\cdot \nabla  \textbf{u}(t^{n+1}), \textbf{u}(t^{n+1}) \right) \\
& \ \ \ \ \ 
\leq C \| \textbf{u}(t^{n+1})-\textbf{u}^n \| \|\textbf{u}(t^{n+1}) \|_1 \|\textbf{u}(t^{n+1}) \|_2 |e_q^{n+1}| \\
& \ \ \ \ \ 
\leq \frac{1}{12T} |e_q^{n+1}|^2 + C\|e_{\textbf{u}}^{n}\|^2 +C\Delta t\int_{t^n}^{t^{n+1} } \|\textbf{u}_t \|^2dt .
\endaligned
\end{equation}

The second to last term on the right hand side of \eqref{e_error_q_inner} can be recast as
\begin{equation}\label{e_error_q_nonlinear5}
\aligned
  - \alpha   \exp( \frac{ t^{n+1} }{T} )& e_q^{n+1} \left(  ( (\nabla \times \textbf{b}^n)  \times \textbf{b}^n, \textbf{u}^{n+1} ) - ( (\nabla \times \textbf{b}(t^{n+1})) \times \textbf{b}(t^{n+1}), \textbf{u}(t^{n+1}) ) \right) \\
= &  \alpha \exp( \frac{t^{n+1}}{T} ) e_q^{n+1}  \left( ( (\nabla \times ( \textbf{b}(t^{n+1}) - 
 \textbf{b}^{n} ) ) \times \textbf{b}^{n}, \textbf{u}(t^{n+1}) \right) \\
& +  \alpha \exp( \frac{t^{n+1}}{T} ) e_q^{n+1} \left( ( (\nabla \times \textbf{b}(t^{n+1} ))  \times ( \textbf{b}(t^{n+1}) -  \textbf{b}^{n} )  , \textbf{u}(t^{n+1}) \right) \\
& -  \alpha \exp( \frac{t^{n+1}}{T} ) e_q^{n+1}  \left( ( \nabla \times \textbf{b}^{n} ) \times \textbf{b}^{n}, e_{\textbf{u}}^{n+1} \right) .
\endaligned
\end{equation}
Thanks to \eqref{e_estimate for trilinear form}, \eqref{e_estimate for trilinear form1} and using the similar procedure  in \eqref{e_error_u_Lorentz_force3}, the first term on the right hand side of \eqref{e_error_q_nonlinear5} can be estimated by
\begin{equation}\label{e_error_q_nonlinear6}
\aligned
   \alpha \exp( \frac{t^{n+1}}{T} )& e_q^{n+1}  \left( ( \nabla \times ( \textbf{b}(t^{n+1}) - 
 \textbf{b}^{n} )  \times \textbf{b}^{n}, \textbf{u}(t^{n+1}) \right) \\
= & - \alpha \exp( \frac{t^{n+1}}{T} ) e_q^{n+1}  \left( ( \nabla \times ( \textbf{u}( t^{n+1} ) \times   \textbf{b}^{n} ), \textbf{b}(t^{n+1}) -  \textbf{b}^{n}  \right) \\
= &  \alpha \exp( \frac{t^{n+1}}{T} ) e_q^{n+1} \left(  ( \textbf{u}( t^{n+1} )  \cdot \nabla  )  \textbf{b}^{n} , \textbf{b}(t^{n+1}) -  \textbf{b}^{n} \right) \\
&  -  \alpha \exp( \frac{t^{n+1}}{T} ) e_q^{n+1} \left(  ( \textbf{b}^{n}  \cdot   \nabla ) \textbf{u}( t^{n+1} ) , \textbf{b}(t^{n+1}) -  \textbf{b}^{n} \right) \\
\leq &  \frac{1}{12 T}  |e_q^{n+1}|^2 + C   \| e_{\textbf{b}}^n \|_1^2 \| e_{\textbf{b}}^n \|^2 + C \| \textbf{u}( t^{n+1} ) \|_2^2 \| \textbf{b}^{n} \| ^2 \Delta t \int_{t^n}^{t^{n+1}} \|\textbf{b}_t \|^2_1 dt .
\endaligned
\end{equation}
For the second term on the right hand side of \eqref{e_error_q_nonlinear5}, we have
\begin{equation}\label{e_error_q_nonlinear7}
\aligned
 & \alpha \exp( \frac{t^{n+1}}{T} ) e_q^{n+1} \left( ( \nabla \times \textbf{b}(t^{n+1} )  \times ( \textbf{b}(t^{n+1}) -  \textbf{b}^{n} )  , \textbf{u}(t^{n+1}) \right) \\
\leq & \frac{1}{12 T}  |e_q^{n+1}|^2 + C \| \textbf{b}( t^{n+1} ) \|_1^2
 \| e_{\textbf{b}}^n \|^2 +  C \| \textbf{u}( t^{n+1} ) \|_2^2 \Delta t \int_{t^n}^{t^{n+1}} \|\textbf{b}_t \|^2 dt .
\endaligned
\end{equation}
Using \eqref{e_norm L4} and \eqref{e_norm curl} and the integration by parts \eqref{e_integration by parts1}, the last term on the right hand side of \eqref{e_error_q_inner} can be bounded by
\begin{equation}\label{e_error_q_nonlinear8}
\aligned
 \alpha   \exp( \frac{t^{n+1}}{T} ) & e_q^{n+1} \left(  ( \nabla \times ( \textbf{b}^n \times \textbf{u}^n ) , \textbf{b}^{n+1} ) - ( \nabla \times ( \textbf{b}(t^{n+1}) \times \textbf{u}(t^{n+1}) ) , \textbf{b}(t^{n+1}) ) \right) \\
 \leq & \alpha \exp( \frac{t^{n+1}}{T} ) e_q^{n+1} \left( \nabla \times ( ( \textbf{b}^n -  \textbf{b}(t^{n+1}) ) \times \textbf{u}^n ),  \textbf{b}(t^{n+1})  \right)  \\
 & +  \alpha \exp( \frac{t^{n+1}}{T} ) e_q^{n+1} \left( \nabla \times ( \textbf{b}(t^{n+1})  \times ( \textbf{u}^n -  \textbf{u}(t^{n+1}) ) ),  \textbf{b}(t^{n+1})  \right) \\
 & +  \alpha \exp( \frac{t^{n+1}}{T} ) e_q^{n+1} \left( \nabla \times ( \textbf{b}^n \times \textbf{u}^n ), e_{\textbf{b}}^{n+1} \right) \\
 \leq &   \alpha \exp( \frac{t^{n+1}}{T} ) e_q^{n+1} \left( \nabla \times ( \textbf{b}^n \times \textbf{u}^n ), e_{\textbf{b}}^{n+1} \right) + \frac{1}{12 T}  |e_q^{n+1}|^2  \\
 & + C \| \textbf{u} ^n \|_1^2 \| e_{ \textbf{b} } ^n\|^2 
 + C  \| e_{ \textbf{u} } ^n\|^2 + C \| \textbf{b}(t^{n+1} )\|^2_2 \Delta t \int_{t^n}^{t^{n+1}} ( \|\textbf{b}_t \|^2_1 + \|\textbf{u}_t \|^2 ) dt .
\endaligned
\end{equation}
Finally, combining  \eqref{e_error_q_nonlinear1}-\eqref{e_error_q_nonlinear8} in \eqref{e_error_q_inner} 
leads to the desired result.
\end{proof}

Now we are in the position to prove  Theorem \ref{thm: error_estimate_ubq}  by using Lemmas \ref{lem: error_estimate_u}-\ref{lem: error_estimate_q}.

{\it Proof of Theorem \ref{thm: error_estimate_ubq}.}

Multiplying both sides of \eqref{lem3.5}  
by $\alpha$ and 
summing up this inequality with \eqref{lem3.4}  
and  \eqref{lem3.6} 
lead to 
\begin{equation}\label{e_error_ubq_final1}
\aligned
&\frac{\| e_{\textbf{u}}^{n+1}\|^2-\|e_{\textbf{u}}^n\|^2}{2\Delta t}+\frac{\| e_{\textbf{u}}^{n+1}-e_{\textbf{u}}^n\|^2}{2\Delta t}+\frac{ \nu }{2} \| \nabla e_{\textbf{u}}^{n+1}\|^2  + \alpha \frac{\| e_{\textbf{b}}^{n+1}\|^2-\|e_{\textbf{b}}^n\|^2}{2\Delta t}  \\
& + \alpha \frac{\| e_{\textbf{b}}^{n+1}-e_{\textbf{b}}^n\|^2}{2\Delta t} + \frac{\alpha \eta }{ 2 } \| \nabla e_{\textbf{b}}^{n+1}\|^2  + \frac{|e_q^{n+1}|^2-|e_q^n|^2}{2\Delta t}+\frac{|e_q^{n+1}-e_q^n|^2}{2\Delta t}+\frac{1}{2T}|e_q^{n+1}|^2  \\
\leq &  C (  \|  \textbf{b}(t^{n+1}) \|^2_2  +\|e_{\textbf{u}}^n\|^2_1)  \|e_{\textbf{u}}^n\|^2   + C( \| e_{ \textbf{b} }^{n} \|_1^2 +  \| \textbf{u} ^n \|_1^2  ) \| e_{ \textbf{b} }^{n} \|^2   \\
 &+  C \Delta t \int_{t^n}^{t^{n+1}} (  \|\textbf{u}_t \|_2^2 + \|\textbf{u}_{tt}\|_{-1}^2 + \|q_{tt} \|^2 ) dt \\
  &+  C \Delta t \int_{t^n}^{t^{n+1}} (   \|\textbf{b}_t \|^2_2 + \|\textbf{b}_{tt}\|_{-1}^2 ) dt . 
\endaligned
\end{equation}
Multiplying \eqref{e_error_ubq_final1} by $2\Delta t$ and summing over $n$, $n=0,1,\ldots,m$, and applying the discrete Gronwall lemma  \ref{lem: gronwall2}, we have 
\begin{equation}\label{e_error_ubq_final2}
\aligned
& \| e_{\textbf{u}}^{m +1}\|^2 +  \| e_{\textbf{b}}^{n+1} \|^2 + |e_q^{m+1}|^2 +
\nu \Delta t \sum\limits_{n=0}^{m} \| \nabla e_{\textbf{u}}^{n+1}\|^2 \\
& + \eta \Delta t \sum\limits_{n=0}^{m} \| \nabla e_{\textbf{b}}^{n+1}\|^2 + \Delta t \sum\limits_{n=0}^{m} |  e_{q}^{n+1} |^2
 + \sum\limits_{n=0}^{m} \| e_{\textbf{u}}^{n+1}-e_{\textbf{u}}^n\|^2 \\
& + \sum\limits_{n=0}^{m} \| e_{\textbf{b}}^{n+1}-e_{\textbf{b}}^n\|^2 + 
\sum\limits_{n=0}^{m} | e_q^{n+1}-e_q^n |^2 \\
\leq &  C ( \|\textbf{u}\|_{H^1(0,T;H^2( {\Omega}) )}^2 +  \|\textbf{u}\|_{H^2(0,T;H^{-1} ( {\Omega}) )}^2 + \| \textbf{u} \|_{L^{\infty}(0,T; \textbf{H}^2(\Omega) )} ^2  ) (\Delta t)^2 \\
& +  C ( \|\textbf{b}\|_{H^1(0,T;H^2( {\Omega}) )}^2 +  \|\textbf{b}\|_{H^2(0,T;H^{-1} {\Omega}) )}^2 )  (\Delta t)^2 \\
& +  C ( \| \textbf{b} \|_{L^{\infty}(0,T; \textbf{H}^2(\Omega) )} ^2 + \|q\|_{H^2(0,T)}^2 )  (\Delta t)^2,
\endaligned
\end{equation}
which concludes the proof of Theorem \ref{thm: error_estimate_ubq}.

\subsection{Error estimates for the pressure}
The main result in this section is the following error estimate for the pressure.

\begin{theorem}\label{thm: error_estimate_p}
Assuming $\textbf{u}\in H^2(0,T;\textbf{L}^{2}(\Omega))\bigcap H^1(0,T;\textbf{H}^2(\Omega))\bigcap L^{\infty}(0,T; \textbf{H}^2(\Omega) )$,   $\textbf{b}\in H^2(0,T;\textbf{L}^{2}(\Omega))\bigcap H^1(0,T;\textbf{H}^2(\Omega))\bigcap L^{\infty}(0,T; \textbf{H}^2(\Omega) )$,  $p\in L^2(0,T; L^2_0(\Omega) )$, 
then for the first-order  scheme \eqref{e_SAV_scheme_first_u}-\eqref{e_SAV_scheme_first_q}, we have 
\begin{equation}\label{e_error_p_L2}
\aligned
&\Delta t\sum\limits_{n=0}^m\|e_{p}^{n+1}\|^2_{L^2(\Omega)/R}  \leq C(\Delta t)^2,  \ \ \ \forall \ 0\leq m\leq N-1,
\endaligned
\end{equation}
where $C$ is a positive constant  independent of $\Delta t$.
\end{theorem}

\begin{proof}
In order to prove the above results, we need to first establish an estimate on $ \| d_te_{\textbf{u}}^{n+1} \|$.    

Thanks to Theorem \ref{thm: error_estimate_ubq}, we have
\begin{equation}\label{e_error_ubH1_boundedness1}
\aligned
& \| e_{\textbf{u}}^{m+1}\|^2 + \| e_{\textbf{b}}^{m+1}\|^2 + \Delta t\sum\limits_{n=0}^m  (  \|\nabla e_{\textbf{u}}^{n+1}\|^2 + \|\nabla e_{\textbf{b}}^{n+1}\|^2 ) \leq C(\Delta t)^2,
\endaligned
\end{equation}
which implies that 
\begin{equation}\label{e_error_ubH1_boundedness2}
\aligned
&\| \textbf{u}^{n+1} \|_{1} \leq C \left( (\Delta t)^{1/2} + \| \textbf{u}(t^{n+1}) \|_1  \right), \ \ 
 \| \textbf{b}^{n+1} \|_{1} \leq C \left( (\Delta t)^{1/2} +  \| \textbf{b}(t^{n+1}) \|_1  \right).
\endaligned
\end{equation}

Taking the inner product of \eqref{e_error_u} with $ Ae_{\textbf{u}}^{n+1}+d_t e_{\textbf{u}}^{n+1} $, we obtain
\begin{equation}\label{e_error_Au}
\aligned
& (1+\nu )  \frac{ \| \nabla e_{\textbf{u}}^{n+1} \|^2 - \| \nabla e_{\textbf{u}}^{n} \|^2 }{ 2\Delta t  } + \| d_t e_{\textbf{u}}^{n+1} \|^2 + \nu \| A e_{\textbf{u}}^{n+1} \|^2 \\
= & \exp( \frac{t^{n+1}}{T} )  \left(  q(t^{n+1}) \textbf{u}(t^{n+1})\cdot \nabla\textbf{u}(t^{n+1}) - q^{n+1}   \textbf{u}^{n}\cdot \nabla  \textbf{u}^{n} ,  Ae_{\textbf{u}}^{n+1}+d_t e_{\textbf{u}}^{n+1} \right) \\
&+ \alpha \exp( \frac{t^{n+1}}{T} )  \left(  q^{n+1} ( \nabla \times \textbf{b}^{n} ) \times \textbf{b}^{n}  - q(t^{n+1}) ( \nabla \times \textbf{b}(t^{n+1}) ) \times \textbf{b}(t^{n+1} ) ,  Ae_{\textbf{u}}^{n+1}+d_t e_{\textbf{u}}^{n+1} \right) \\
& + ( \textbf{R}_{\textbf{u}}^{n+1} , Ae_{\textbf{u}}^{n+1}+d_t e_{\textbf{u}}^{n+1}  ) .
\endaligned
\end{equation}
For the first term on the right hand side of \eqref{e_error_Au}, we have
\begin{equation}\label{e_error_Au_nonlinear1}
\aligned
  \exp( \frac{t^{n+1}}{T} ) & \left(  q(t^{n+1}) \textbf{u}(t^{n+1})\cdot \nabla\textbf{u}(t^{n+1}) - q^{n+1}  \textbf{u}^{n}\cdot \nabla  \textbf{u}^{n} ,  Ae_{\textbf{u}}^{n+1}+d_t e_{\textbf{u}}^{n+1} \right) \\
= & - \exp( \frac{t^{n+1}}{T} ) e_q^{n+1} \left(  ( \textbf{u}^{n}\cdot \nabla ) \textbf{u}^{n}  ,  Ae_{\textbf{u}}^{n+1}+d_t e_{\textbf{u}}^{n+1} \right) \\
& +  \left(  (  \textbf{u}(t^{n+1}) -  \textbf{u}^{n} ) \cdot \nabla \textbf{u}(t^{n+1}) ,  Ae_{\textbf{u}}^{n+1}+d_t e_{\textbf{u}}^{n+1} \right) \\
& +  \left(   \textbf{u}^{n} \cdot \nabla ( \textbf{u}(t^{n+1}) -  \textbf{u}^{n} )  ,  Ae_{\textbf{u}}^{n+1}+d_t e_{\textbf{u}}^{n+1} \right) . 
\endaligned
\end{equation}
Thanks to \eqref{e_estimate for trilinear form1} and \eqref{e_error_ubH1_boundedness2}, the first term on the right hand side of \eqref{e_error_Au_nonlinear1} can be bounded by
\begin{equation}\label{e_error_Au_nonlinear2}
\aligned
 - \exp( \frac{t^{n+1}}{T} )& e_q^{n+1} \left(   \textbf{u}^{n}\cdot \nabla  \textbf{u}^{n}  ,  Ae_{\textbf{u}}^{n+1}+d_t e_{\textbf{u}}^{n+1} \right) \\
= & - \exp( \frac{t^{n+1}}{T} ) e_q^{n+1} \left(  \textbf{u}^{n}\cdot \nabla  e_{ \textbf{u} }^{n}  ,  Ae_{\textbf{u}}^{n+1}+d_t e_{\textbf{u}}^{n+1} \right) \\
& -  \exp( \frac{t^{n+1}}{T} ) e_q^{n+1} \left(  (\textbf{u}^{n}\cdot \nabla  \textbf{u}(t^n)  ,  Ae_{\textbf{u}}^{n+1}+d_t e_{\textbf{u}}^{n+1} \right) \\
\leq & C | e_q^{n+1}| \| \textbf{u}^{n} \|^{1/2} \| \nabla \textbf{u}^{n} \|^{1/2}  \| e_{\textbf{u}}^{n} \|^{1/2}  \| A e_{ \textbf{u} }^{n} \|^{1/2} \| Ae_{\textbf{u}}^{n+1}+d_t e_{\textbf{u}}^{n+1} \| \\
& + C | e_q^{n+1}| \|  \textbf{u}^{n} \|_1  \| \textbf{u}(t^n) \|_2^2 \| Ae_{\textbf{u}}^{n+1}+d_t e_{\textbf{u}}^{n+1} \| \\
\leq & \frac{1}{12} \| d_t e_{\textbf{u}}^{n+1} \|^2 + \frac{ \nu } {24} \| A e_{\textbf{u}}^{n+1} \|^2  + \frac{ \nu } {8} \| A e_{\textbf{u}}^{n} \|^2  \\
& +  C ( \Delta t + \| \textbf{u}(t^{n}) \|_1^2 ) \| e_{\textbf{u} }^n \|^2  +  C ( \Delta t + \| \textbf{u}(t^{n}) \|_1^2 ) |e_q^{n+1}|^2  .
\endaligned
\end{equation}
The second term on the right hand side of \eqref{e_error_Au_nonlinear1} can be estimated by
\begin{equation}\label{e_error_Au_nonlinear3}
\aligned
 \big(  (  \textbf{u}(t^{n+1}) -  \textbf{u}^{n} ) &\cdot \nabla \textbf{u}(t^{n+1}) ,  Ae_{\textbf{u}}^{n+1}+d_t e_{\textbf{u}}^{n+1} \big) \\
\leq  & C \| \textbf{u}(t^{n+1}) -  \textbf{u}^{n} \|_1 \|  \textbf{u}(t^{n+1}) \|_2 \| Ae_{\textbf{u}}^{n+1}+d_t e_{\textbf{u}}^{n+1} \|  \\
\leq &   \frac{1}{12} \| d_t e_{\textbf{u}}^{n+1} \|^2 + \frac{ \nu } {24} \| A e_{\textbf{u}}^{n+1} \|^2 + C \| e_{ \textbf{u} }^n \|_1^2  \\
& + C \|  \textbf{u}(t^{n+1}) \|_2 ^2  \Delta t \int_{t^n}^{t^{n+1}}  \|\textbf{u}_t \|_1^2 dt .
\endaligned
\end{equation}
Using \eqref{e_estimate for trilinear form1} and \eqref{e_error_ubH1_boundedness2}, the last term on the right hand side of \eqref{e_error_Au_nonlinear1} can be controlled by
\begin{equation}\label{e_error_Au_nonlinear4}
\aligned
\big(  \textbf{u}^{n} \cdot \nabla ( \textbf{u}(t^{n+1}) -  \textbf{u}^{n} ) & ,  Ae_{\textbf{u}}^{n+1}+d_t e_{\textbf{u}}^{n+1} \big) \\ 
= &  \left(   \textbf{u}^{n} \cdot \nabla ( \textbf{u}(t^{n+1}) -  \textbf{u}( t^{n} ) )  ,  Ae_{\textbf{u}}^{n+1}+d_t e_{\textbf{u}}^{n+1} \right) \\
& -  \left(   \textbf{u}^{n} \cdot \nabla e_ { \textbf{u} }^n  ,  Ae_{\textbf{u}}^{n+1}+d_t e_{\textbf{u}}^{n+1} \right) \\
\leq & C \| \textbf{u}^{n} \|_1 \| \textbf{u}(t^{n+1}) -  \textbf{u}( t^{n} ) \|_2 \| Ae_{\textbf{u}}^{n+1}+d_t e_{\textbf{u}}^{n+1} \| \\
& + C \| \textbf{u}^{n} \|_1^{1/2} \| \textbf{u}^{n} \|_0^{1/2} \| A e_{ \textbf{u} }^n \|^{1/2} \| e_{ \textbf{u} }^n \|^{1/2} \| Ae_{\textbf{u}}^{n+1}+d_t e_{\textbf{u}}^{n+1} \| \\
\leq &  \frac{1}{12} \| d_t e_{\textbf{u}}^{n+1} \|^2 + \frac{ \nu } {24} \| A e_{\textbf{u}}^{n+1} \|^2  +  C ( \Delta t + \| \textbf{u}(t^{n+1}) \|_1^2 ) \| e_{\textbf{u} }^n \|^2 \\ 
&  + \frac{ \nu } {8} \| A e_{\textbf{u}}^{n} \|^2  + C ( \Delta t + \| \textbf{u}(t^{n}) \|_1^2 )   \Delta t \int_{t^n}^{t^{n+1}}  \|\textbf{u}_t \|_2^2 dt .
\endaligned
\end{equation}
For the second term on the right hand side of \eqref{e_error_Au}, we have
\begin{equation}\label{e_error_Au_nonlinear5}
\aligned
 \alpha \exp( \frac{t^{n+1}}{T} ) & \left(  q^{n+1} ( \nabla \times \textbf{b}^{n} ) \times \textbf{b}^{n}  - q(t^{n+1}) ( \nabla \times \textbf{b}(t^{n+1}) ) \times \textbf{b}(t^{n+1} ) ,  Ae_{\textbf{u}}^{n+1}+d_t e_{\textbf{u}}^{n+1} \right) \\
= &  \alpha \exp( \frac{t^{n+1}}{T} ) e_q^{n+1}  \left(   ( \nabla \times \textbf{b}^{n} ) \times \textbf{b}^{n}  ,  Ae_{\textbf{u}}^{n+1}+d_t e_{\textbf{u}}^{n+1} \right) \\
& +  \alpha  \left(   ( \nabla \times ( \textbf{b}^{n} - \textbf{b}(t^{n+1}) ) ) \times \textbf{b}^{n}  ,  Ae_{\textbf{u}}^{n+1}+d_t e_{\textbf{u}}^{n+1} \right) \\
& + \alpha  \left(   ( \nabla \times \textbf{b}(t^{n+1})  ) \times  ( \textbf{b}^{n} - \textbf{b}(t^{n+1}) )  ,  Ae_{\textbf{u}}^{n+1}+d_t e_{\textbf{u}}^{n+1} \right) . 
\endaligned
\end{equation}
Thanks to \eqref{e_norm L_infty} and  and \eqref{e_error_ubH1_boundedness2},  the first term on the right hand side of \eqref{e_error_Au_nonlinear5} can be bounded by
\begin{equation}\label{e_error_Au_nonlinear6}
\aligned
  \alpha \exp( \frac{t^{n+1}}{T} )& e_q^{n+1}  \left(   ( \nabla \times \textbf{b}^{n} ) \times \textbf{b}^{n}  ,  Ae_{\textbf{u}}^{n+1}+d_t e_{\textbf{u}}^{n+1} \right) \\
= &  \alpha \exp( \frac{t^{n+1}}{T} ) e_q^{n+1}  \left(   ( \nabla \times \textbf{b}^{n} ) \times e_{ \textbf{b} } ^{n}  ,  Ae_{\textbf{u}}^{n+1}+d_t e_{\textbf{u}}^{n+1} \right)  \\
& +  \alpha \exp( \frac{t^{n+1}}{T} ) e_q^{n+1}  \left(   ( \nabla \times \textbf{b}^{n} ) \times \textbf{b} ( t^{n} )  ,  Ae_{\textbf{u}}^{n+1}+d_t e_{\textbf{u}}^{n+1} \right) \\
\leq & C \|  \nabla \times \textbf{b}^{n} \| \|  e_{ \textbf{b} }^n \|_1^{1/2} \|  e_{ \textbf{b} }^n \|_2^{1/2} \|  Ae_{\textbf{u}}^{n+1}+d_t e_{\textbf{u}}^{n+1} \|  \\
& + C | e_q^{n+1} | \|  \nabla \times \textbf{b}^{n} \| \|  \textbf{b} ( t^{n} )  \|_2 \|  Ae_{\textbf{u}}^{n+1}+d_t e_{\textbf{u}}^{n+1} \| \\
\leq &  \frac{1}{12} \| d_t e_{\textbf{u}}^{n+1} \|^2 + \frac{ \nu } {24} \| A e_{\textbf{u}}^{n+1} \|^2  + \frac{ \eta } {8} \| \Delta e_{\textbf{b}}^{n} \|^2 \\
& +  C ( \Delta t + \| \textbf{b}(t^{n}) \|_1^2 ) \| e_{\textbf{b}}^{n} \|_1^2  +  C ( \Delta t + \| \textbf{b}(t^{n}) \|_1^2 ) | e_q^{n+1} |^2 .
\endaligned
\end{equation}
The last two terms on the right hand side of \eqref{e_error_Au_nonlinear5} can be estimated by
\begin{equation}\label{e_error_Au_nonlinear7}
\aligned
   \alpha  \big(   ( \nabla \times & ( \textbf{b}^{n} - \textbf{b}(t^{n+1}) ) ) \times \textbf{b}^{n}  ,  Ae_{\textbf{u}}^{n+1}+d_t e_{\textbf{u}}^{n+1} \big) \\
& + \alpha  \left(   ( \nabla \times \textbf{b}(t^{n+1})  ) \times  ( \textbf{b}^{n} - \textbf{b}(t^{n+1}) )  ,  Ae_{\textbf{u}}^{n+1}+d_t e_{\textbf{u}}^{n+1} \right)  \\
\leq & C \|  e_{ \textbf{b} }^n + \textbf{b}(t^{n} ) - \textbf{b}(t^{n+1})   \| _1 
\| e_{ \textbf{b} }^n  \|_1^{1/2} \| e_{ \textbf{b} }^n   \|_2^{1/2}  \| Ae_{\textbf{u}}^{n+1}+d_t e_{\textbf{u}}^{n+1} \| \\
& + C \|  e_{ \textbf{b} }^n + \textbf{b}(t^{n} ) - \textbf{b}(t^{n+1})   \| _1 
 \|  \textbf{b}(t^{n} )  \|_2  \| Ae_{\textbf{u}}^{n+1}+d_t e_{\textbf{u}}^{n+1} \| \\
& + C \|  \nabla \times \textbf{b}(t^{n+1}) \|_{L^4} \|  \textbf{b}^{n} - \textbf{b}(t^{n+1}) \| _{L^4} \| Ae_{\textbf{u}}^{n+1}+d_t e_{\textbf{u}}^{n+1} \| \\
\leq &  \frac{1}{12} \| d_t e_{\textbf{u}}^{n+1} \|^2 + \frac{ \nu } {24} \| A e_{\textbf{u}}^{n+1} \|^2  + \frac{ \eta } {8} \| \Delta e_{\textbf{b}}^{n} \|^2 \\
& + C \| e_{\textbf{b}}^{n} \|_1^2  +  C \|  \textbf{b}(t^{n+1}) \|_2 ^2 \Delta t \int_{t^n}^{t^{n+1}}  \|\textbf{b}_t \|_1^2 dt .
\endaligned
\end{equation}
For the last term on the right hand side of \eqref{e_error_Au}, we have 
\begin{equation}\label{e_error_Au_Ru}
\aligned
&  ( \textbf{R}_{\textbf{u}}^{n+1} , Ae_{\textbf{u}}^{n+1}+d_t e_{\textbf{u}}^{n+1}  ) \leq 
 \frac{1}{12} \| d_t e_{\textbf{u}}^{n+1} \|^2 + \frac{ \nu } {24} \| A e_{\textbf{u}}^{n+1} \|^2  
+ C \Delta t\int_{t^n}^{t^{n+1}}\|\textbf{u}_{tt}\|^2dt .
\endaligned
\end{equation}
Combining \eqref{e_error_Au} with \eqref{e_error_Au_nonlinear1}-\eqref{e_error_Au_Ru}, we have
\begin{equation}\label{e_error_Au_combine}
\aligned
& (1+\nu )  \frac{ \| \nabla e_{\textbf{u}}^{n+1} \|^2 - \| \nabla e_{\textbf{u}}^{n} \|^2 }{ 2\Delta t  } + \frac{1}{2} \| d_t e_{\textbf{u}}^{n+1} \|^2 + \frac{3 \nu} {4}  \| A e_{\textbf{u}}^{n+1} \|^2 \\
& \ \ \ \ \ \ 
\leq   \frac{  \eta } {4} \| \Delta e_{\textbf{b}}^{n} \|^2 + \frac{ \nu} {4}  \| A e_{\textbf{u}}^{n} \|^2 +C ( \Delta t + \| \textbf{u}(t^{n}) \|_1^2 ) \| e_{\textbf{u} }^n \|^2_1 + C ( \Delta t + \| \textbf{b}(t^{n}) \|_1^2 ) \| e_{\textbf{b} }^n \|^2_1 \\
& \ \ \ \ \ \  +C ( \Delta t + \| \textbf{u}(t^{n}) \|_1^2 + \| \textbf{b}(t^{n}) \|_1^2 ) |e_q^{n+1}|^2  \\
& \ \ \ \ \ \   + C  \Delta t \int_{t^n}^{t^{n+1}}  ( \|\textbf{u}_t \|_2^2 + \|\textbf{u}_{tt}\|^2 + \|\textbf{b}_t \|_1^2 ) dt .
\endaligned
\end{equation}
Next we shall balance the first term on the right hand side of \eqref{e_error_Au_combine}
by using the error equation \eqref{e_error_b} for magnetic field. We proceed as follows.

Taking the inner product of \eqref{e_error_b} with $ -\Delta e_{\textbf{b}}^{n+1}+d_t e_{\textbf{b}}^{n+1} $, we obtain
\begin{equation}\label{e_error_Ab}
\aligned
(1+\eta ) & \frac{ \| \nabla e_{\textbf{b}}^{n+1} \|^2 - \| \nabla e_{\textbf{b}}^{n} \|^2 }{ 2\Delta t  } + \| d_t e_{\textbf{b}}^{n+1} \|^2 + \eta \| \Delta e_{\textbf{b}}^{n+1} \|^2 \\
=&  \exp( \frac{t^{n+1}}{T} ) q(t^{n+1}) \left(  \nabla \times ( \textbf{b}(t^{n+1}) \times \textbf{u}(t^{n+1}) )  , -\Delta e_{\textbf{b}}^{n+1}+d_t e_{\textbf{b}}^{n+1} \right)  \\
&  - \exp( \frac{t^{n+1}}{T} ) q^{n+1} \left(   \nabla \times ( \textbf{b}^n \times \textbf{u}^n ) , -\Delta e_{\textbf{b}}^{n+1}+d_t e_{\textbf{b}}^{n+1} \right) \\
& + (\textbf{R}_{\textbf{b}}^{n+1}, -\Delta e_{\textbf{b}}^{n+1}+d_t e_{\textbf{b}}^{n+1} ) .
\endaligned
\end{equation}
The first two terms on the right hand side of \eqref{e_error_Ab} can be recast as
\begin{equation}\label{e_error_Ab_nonlinear1}
\aligned
 \exp( \frac{t^{n+1}}{T} ) &q(t^{n+1}) \left(  \nabla \times ( \textbf{b}(t^{n+1}) \times \textbf{u}(t^{n+1}) )  , -\Delta e_{\textbf{b}}^{n+1}+d_t e_{\textbf{b}}^{n+1} \right)  \\
&  - \exp( \frac{t^{n+1}}{T} ) q^{n+1} \left(   \nabla \times ( \textbf{b}^n \times \textbf{u}^n ) , -\Delta e_{\textbf{b}}^{n+1}+d_t e_{\textbf{b}}^{n+1} \right) \\
= & \left(  \nabla \times [ ( \textbf{b}(t^{n+1}) - \textbf{b}^n ) \times \textbf{u}(t^{n+1}) ] , -\Delta e_{\textbf{b}}^{n+1}+d_t e_{\textbf{b}}^{n+1} \right) \\
& + \left(  \nabla \times [ \textbf{b}^n  \times ( \textbf{u}(t^{n+1}) - \textbf{u}^n ) ] , -\Delta e_{\textbf{b}}^{n+1}+d_t e_{\textbf{b}}^{n+1} \right)  \\
& -  \exp( \frac{t^{n+1}}{T} ) e_q^{n+1} \left( \nabla \times ( \textbf{b}^n \times \textbf{u}^n ), -\Delta e_{\textbf{b}}^{n+1}+d_t e_{\textbf{b}}^{n+1} \right) .
\endaligned
\end{equation}
Noting \eqref{e_cross_product2_plus} and \eqref{e_estimate for trilinear form},  the first term on the right hand side of \eqref{e_error_Ab_nonlinear1} can be bounded by
\begin{equation}\label{e_error_Ab_nonlinear2}
\aligned
\big(  \nabla \times & [ ( \textbf{b}(t^{n+1}) - \textbf{b}^n ) \times \textbf{u}(t^{n+1}) ] , -\Delta e_{\textbf{b}}^{n+1}+d_t e_{\textbf{b}}^{n+1} \big) \\
\leq & C \| \textbf{b}(t^{n+1}) - \textbf{b}^n \|_1 \| \textbf{u}(t^{n+1}) \|_2 \| d_t e_{\textbf{b}}^{n+1} - \Delta e_{\textbf{b}}^{n+1} \| \\
\leq & \frac{1}{8} \| d_t e_{\textbf{b}}^{n+1}  \|^2 + \frac{ \eta }{ 16} \| \Delta e_{\textbf{b}}^{n+1} \| ^2 + C \| e_{ \textbf{b} } ^n \|^2_1 \\
& + C \| \textbf{u}(t^{n+1}) \|_2^2  \Delta t \int_{t^n}^{t^{n+1}}   \|\textbf{b}_t \|_1^2 dt .
\endaligned
\end{equation}
For the second term on the right hand side of \eqref{e_error_Ab_nonlinear1}, we have
\begin{equation}\label{e_error_Ab_nonlinear3}
\aligned
\big(  \nabla \times &[ \textbf{b}^n  \times ( \textbf{u}(t^{n+1}) - \textbf{u}^n ) ] , -\Delta e_{\textbf{b}}^{n+1}+d_t e_{\textbf{b}}^{n+1} \big)  \\
= &  \left(  \nabla \times [ e_ { \textbf{b} }^n  \times ( \textbf{u}(t^{n+1}) - \textbf{u}^n ) ] , -\Delta e_{\textbf{b}}^{n+1}+d_t e_{\textbf{b}}^{n+1} \right) \\
& +  \left(  \nabla \times [ \textbf{b} (t^n)  \times ( \textbf{u}(t^{n+1}) - \textbf{u}^n ) ] , -\Delta e_{\textbf{b}}^{n+1}+d_t e_{\textbf{b}}^{n+1} \right) \\
\leq & C \| e_ { \textbf{b} }^n \|_{1}^{1/2} \| e_ { \textbf{b} }^n \|_{2}^{1/2} \| \textbf{u}(t^{n+1}) - \textbf{u}^n \|_1 \| d_t e_{\textbf{b}}^{n+1} - \Delta e_{\textbf{b}}^{n+1} \| \\
& +  C \| \textbf{b} (t^n) \|_{2} \| \textbf{u}(t^{n+1}) - \textbf{u}^n \|_1 \| d_t e_{\textbf{b}}^{n+1} - \Delta e_{\textbf{b}}^{n+1} \| \\
\leq &  \frac{1}{8} \| d_t e_{\textbf{b}}^{n+1}  \|^2 + \frac{ \eta }{ 16} \| \Delta e_{\textbf{b}}^{n+1} \| ^2 + \frac{ \eta }{ 8} \| \Delta e_{\textbf{b}}^{n} \| ^2 + C \| e_{ \textbf{b} } ^n \|^2_1 \\
& + C \| \textbf{b}(t^{n}) \|_2^2  \Delta t \int_{t^n}^{t^{n+1}}   \|\textbf{u}_t \|_1^2 dt .
\endaligned
\end{equation}
Thanks to \eqref{e_cross_product2_plus} and \eqref{e_estimate for trilinear form}, the last term on the right hand side of \eqref{e_error_Ab_nonlinear1} can be 
\begin{equation}\label{e_error_Ab_nonlinear4}
\aligned
  -  \exp( \frac{t^{n+1}}{T} )& e_q^{n+1} \left( \nabla \times ( \textbf{b}^n \times \textbf{u}^n ), -\Delta e_{\textbf{b}}^{n+1}+d_t e_{\textbf{b}}^{n+1} \right) \\
= &  -  \exp( \frac{t^{n+1}}{T} ) e_q^{n+1} \left( \nabla \times ( e_{ \textbf{b} }^n \times \textbf{u}^n ), -\Delta e_{\textbf{b}}^{n+1}+d_t e_{\textbf{b}}^{n+1} \right) \\
& -  \exp( \frac{t^{n+1}}{T} ) e_q^{n+1} \left( \nabla \times (  \textbf{b} (t^n) \times \textbf{u}^n ), -\Delta e_{\textbf{b}}^{n+1}+d_t e_{\textbf{b}}^{n+1} \right) \\
\leq & C |e_q^{n+1}|  \| e_ { \textbf{b} }^n \|_{1}^{1/2} \| e_ { \textbf{b} }^n \|_{2}^{1/2} \|  \textbf{u} ^n \|_{1} \| d_t e_{\textbf{b}}^{n+1} - \Delta e_{\textbf{b}}^{n+1} \| \\
& + C |e_q^{n+1}|  \| \textbf{b} (t^n) \|_{2} \|  \textbf{u} ^n \|_{1} \| d_t e_{\textbf{b}}^{n+1} - \Delta e_{\textbf{b}}^{n+1} \| \\
\leq &   \frac{1}{8} \| d_t e_{\textbf{b}}^{n+1}  \|^2 + \frac{ \eta }{ 16} \| \Delta e_{\textbf{b}}^{n+1} \| ^2 + \frac{ \eta }{ 8} \| \Delta e_{\textbf{b}}^{n} \| ^2 \\
& + C ( \Delta t + \| \textbf{u}(t^{n}) \|_1^2  ) \| e_{ \textbf{b} } ^n \|^2_1 + C ( \Delta t + \| \textbf{u}(t^{n}) \|_1^2  ) |e_q^{n+1}| ^2 .
\endaligned
\end{equation}
For the last term on the right hand side of \eqref{e_error_Ab}, we have 
\begin{equation}\label{e_error_Ab_Rb}
\aligned
&  (\textbf{R}_{\textbf{b}}^{n+1}, -\Delta e_{\textbf{b}}^{n+1}+d_t e_{\textbf{b}}^{n+1} ) \leq 
 \frac{1}{8} \| d_t e_{\textbf{b}}^{n+1} \|^2 + \frac{ \eta } {16} \| \Delta  e_{\textbf{b}}^{n+1} \|^2  
+ C \Delta t\int_{t^n}^{t^{n+1}}\|\textbf{b}_{tt}\|^2dt .
\endaligned
\end{equation}
Combining \eqref{e_error_Ab} with \eqref{e_error_Ab_nonlinear1}-\eqref{e_error_Ab_Rb}, we obtain
\begin{equation}\label{e_error_Ab_combine}
\aligned
 (1+\eta ) & \frac{ \| \nabla e_{\textbf{b}}^{n+1} \|^2 - \| \nabla e_{\textbf{b}}^{n} \|^2 }{ 2\Delta t  } + \frac{1}{2} \| d_t e_{\textbf{b}}^{n+1} \|^2 + \frac{3 \eta } {4} \| \Delta e_{\textbf{b}}^{n+1} \|^2 \\
\leq &  \frac{ \eta }{ 4 } \| \Delta e_{\textbf{b}}^{n} \| ^2 + C ( \Delta t + \| \textbf{u}(t^{n}) \|_1^2  ) \| e_{ \textbf{b} } ^n \|^2_1 + C ( \Delta t + \| \textbf{u}(t^{n}) \|_1^2  ) |e_q^{n+1}| ^2 \\
& + C \Delta t \int_{t^n}^{t^{n+1}}  ( \|\textbf{u}_t \|_1^2 + \|\textbf{b}_t \|_1^2  + \|\textbf{b}_{tt}\|^2 ) dt .
\endaligned
\end{equation}
Summing up \eqref{e_error_Ab_combine} with \eqref{e_error_Au_combine} leads to
\begin{equation}\label{e_error_AuAb_final1}
\aligned
& (1+\nu )  \frac{ \| \nabla e_{\textbf{u}}^{n+1} \|^2 - \| \nabla e_{\textbf{u}}^{n} \|^2 }{ 2\Delta t  } + \frac{1}{2} \| d_t e_{\textbf{u}}^{n+1} \|^2 + \frac{3 \nu} {4}  \| A e_{\textbf{u}}^{n+1} \|^2 \\
& + (1+\eta )  \frac{ \| \nabla e_{\textbf{b}}^{n+1} \|^2 - \| \nabla e_{\textbf{b}}^{n} \|^2 }{ 2\Delta t  } + \frac{1}{2} \| d_t e_{\textbf{b}}^{n+1} \|^2 + \frac{3 \eta } {4} \| \Delta e_{\textbf{b}}^{n+1} \|^2 \\  
\leq &  \frac{  \eta } {2} \| \Delta e_{\textbf{b}}^{n} \|^2 + \frac{ \nu} {4}  \| A e_{\textbf{u}}^{n} \|^2 + C ( \Delta t + \| \textbf{u}(t^{n}) \|_1^2  ) \| e_{\textbf{u} }^n \|^2_1 \\
&   +  C ( \Delta t + \| \textbf{u}(t^{n}) \|_1^2  + \| \textbf{b}(t^{n}) \|_1^2 )(  \| e_{\textbf{b} }^n \|^2_1 +  |e_q^{n+1}|^2 ) \\
& 
 + C  \Delta t \int_{t^n}^{t^{n+1}}  ( \|\textbf{u}_t \|_2^2 + \|\textbf{u}_{tt}\|^2 + \|\textbf{b}_t \|_1^2 +  \|\textbf{b}_{tt}\|^2 ) dt .
\endaligned
\end{equation}
Multiplying \eqref{e_error_AuAb_final1} by $2\Delta t$ and summing over $n$, $n=0,2,\ldots,m$, and applying the discrete Gronwall lemma  \ref{lem: gronwall2}, we obtain
\begin{equation}\label{e_error_AuAb_final2}
\aligned
\| \nabla e_{\textbf{u}}^{m+1} \|^2& + \Delta t \sum\limits_{n=0}^{m}  \| d_t e_{\textbf{u}}^{n+1} \|^2 + \nu  \Delta t \sum\limits_{n=0}^{m} \| A e_{\textbf{u}}^{n+1} \|^2 \\
& + \| \nabla e_{\textbf{b}}^{m+1} \|^2 + \Delta t \sum\limits_{n=0}^{m}  \| d_t e_{\textbf{b}}^{n+1} \|^2  +  \eta  \Delta t \sum\limits_{n=0}^{m} \| \Delta e_{\textbf{b}}^{n+1} \|^2 \\
\leq &  C ( \Delta t + \| \textbf{u}(t^{n}) \|_1^2  + \| \textbf{b}(t^{n}) \|_1^2 )  \Delta t \sum\limits_{n=0}^{m} ( \| e_{\textbf{u} }^n \|^2_1 + \| e_{\textbf{b} }^n \|^2_1 )  \\
& +  C \Delta t \sum\limits_{n=0}^{m} | e_q^{n+1} |^2 + C (\Delta t)^2 .
\endaligned
\end{equation}
Combining the above estimate with Theorem \ref{thm: error_estimate_ubq}, we finally obtain  
\begin{equation}\label{e_error_AuAb_final3}
\aligned
\Delta t& \sum\limits_{n=0}^{m}  \| d_t e_{\textbf{u}}^{n+1} \|^2 + \| \nabla e_{\textbf{u}}^{m+1} \|^2 +  \nu  \Delta t \sum\limits_{n=0}^{m} \| A e_{\textbf{u}}^{n+1} \|^2 + \| \nabla e_{\textbf{b}}^{m+1} \|^2  \\
& + \Delta t \sum\limits_{n=0}^{m}  \| d_t e_{\textbf{b}}^{n+1} \|^2  +  \eta  \Delta t \sum\limits_{n=0}^{m} \| \Delta e_{\textbf{b}}^{n+1} \|^2 
\leq   C (\Delta t)^2 .
\endaligned
\end{equation}

We are now in position to prove the pressure estimate. 
 Taking the inner product of \eqref{e_error_u} with $\textbf{v}\in \textbf{H}^1_0(\Omega)$, we obtain
\begin{equation}\label{e_error_p1}
\aligned
(\nabla e_p^{n+1},\textbf{v})=&-( d_t e_{\textbf{u} }^{n+1} , \textbf{v})+\nu(\Delta e_{\textbf{u}}^{n+1},\textbf{v})+(\textbf{R}_{\textbf{u}}^{n+1},\textbf{v})\\
&+ \exp( \frac{t^{n+1}}{T} )  \left( q(t^{n+1}) ( \textbf{u}(t^{n+1})\cdot \nabla ) \textbf{u}(t^{n+1}) - q^{n+1} ( \textbf{u}^{n}\cdot \nabla ) \textbf{u}^{n}, \textbf{v}\right) \\
& +  \alpha \exp( \frac{t^{n+1}}{T} )  \left( q^{n+1} ( \nabla \times \textbf{b}^{n} ) \times \textbf{b}^{n} - q(t^{n+1}) ( \nabla \times \textbf{b}(t^{n+1}) ) \times \textbf{b}(t^{n+1} ) , 
\textbf{v} \right) . 
\endaligned
\end{equation}
 We derive from
\begin{equation}\label{e_error_p2}
\aligned
&\| e_p^{n+1} \|_{L^2(\Omega)/\mathbb{R}} \leq \sup_{\textbf{v} \in \textbf{H}^1_0(\Omega)} \frac{
(\nabla e_p^{n+1},\textbf{v}) }{ \|\nabla \textbf{v} \| },
\endaligned
\end{equation} 
and \eqref{e_error_u_nonlinear_convective}-\eqref{e_error_u_nonlinear_convective2} that, for all $\textbf{v}\in \textbf{H}^1_0(\Omega)$, 
\begin{equation}\label{e_error_p3}
\aligned
 \exp( \frac{t^{n+1}}{T} )&  \left( q(t^{n+1}) ( \textbf{u}(t^{n+1})\cdot \nabla ) \textbf{u}(t^{n+1}) - q^{n+1} ( \textbf{u}^{n}\cdot \nabla ) \textbf{u}^{n}, \textbf{v}\right) \\
=&\frac{q(t^{n+1}) }{ \exp(  -\frac{t^{n+1}}{T} ) }\left( ( \textbf{u}(t^{n+1})-\textbf{u}^{n} )\cdot \nabla \textbf{u}(t^{n+1}),  \textbf{v} \right) -\frac{e_q^{n+1}}{ \exp(  -\frac{t^{n+1}}{T} ) }\left((\textbf{u}^n \cdot \nabla)\textbf{u}^n,  \textbf{v} \right) \\
&+\frac{q(t^{n+1}) }{ \exp(  -\frac{t^{n+1}}{T} ) }\left(  \textbf{u}^{n}\cdot \nabla (\textbf{u}(t^{n+1})-\textbf{u}^{n} ),  \textbf{v} \right) \\
\leq & C(\|e_{ \textbf{u} }^{n}\|_1 +\| \int_{t^n}^{t^{n+1}}\textbf{u}_tdt \|_1+ |e_q^{n+1}| )  \|\nabla \textbf{v} \| ,
\endaligned
\end{equation}
and for the last term on the right hand side of \eqref{e_error_p1}, by using \eqref{e_error_u_Lorentz_force}-\eqref{e_error_u_Lorentz_force5}, we have
\begin{equation}\label{e_error_p4}
\aligned
  \alpha \exp( \frac{t^{n+1}}{T} ) & \left( q^{n+1} ( \nabla \times \textbf{b}^{n} ) \times \textbf{b}^{n} - q(t^{n+1}) ( \nabla \times \textbf{b}(t^{n+1}) ) \times \textbf{b}(t^{n+1} ) , 
\textbf{v} \right) \\
= &  \alpha \exp( \frac{t^{n+1}}{T} ) e_q^{n+1}  \left( ( \nabla \times \textbf{b}^{n} ) \times \textbf{b}^{n}, e_{\textbf{u}}^{n+1} \right) + \alpha \left(  \nabla \times (\textbf{b}^{n}- \textbf{b}(t^{n+1}) ) \times \textbf{b}^{n}, e_{\textbf{u}}^{n+1} \right) \\
& + \alpha \left(  ( \nabla \times \textbf{b} (t^{n+1}) ) \times ( \textbf{b}^{n} - \textbf{b}(t^{n+1}) ), e_{\textbf{u}}^{n+1} \right) \\
\leq &  C( \|e_{ \textbf{b} }^{n}\|_1 + \| \textbf{b}^n \| \| \int_{t^n}^{t^{n+1}}\textbf{b}_tdt \|_1+ |e_q^{n+1}| )  \|\nabla \textbf{v} \| .
\endaligned
\end{equation}
Finally thanks to Theorem \ref{thm: error_estimate_ubq} and \eqref{e_error_AuAb_final3}, we can derive from the above that 
\begin{equation*}\label{e_error_p5}
\aligned
&\Delta t \sum\limits_{n=0}^m \|e_p^{n+1}\|^2_{L^2(\Omega)/\mathbb{R}} \leq 
C\Delta t \sum\limits_{n=0}^m \left( \| d_te_{\textbf{u}}^{n+1}\|^2+ \| \nabla e_{\textbf{u}}^{n+1} \|^2 \right. \\
&\ \ \ \ \ 
\left. +  \| e_{\textbf{u}}^{n} \|_1^2+\|e_{\textbf{b}}^n\|_1^2 + |e_q^{n+1}|^2 \right) +C (\Delta t)^2 \int_{t^0}^{t^{m+1}} \|\textbf{b}_t\|_1^2dt  \\
&\ \ \ \ \  +C (\Delta t)^2 \int_{t^0}^{t^{m+1}} (\|\textbf{u}_t\|_1^2+\|\textbf{u}_{tt}\|_{-1}^2 ) dt   \leq  C(\Delta t)^2.
\endaligned
\end{equation*}
The proof  is complete.
\end{proof}

 \section{Numerical experiments} 
 In this section we provide some numerical experiments to validate the SAV schemes developed in the previous sections.  
 
Although we only discussed semi-discretization in time in the previous sections, the IMEX
SAV schemes can be coupled with any compatible spatial discretization. More precisely, let $ \textbf{X}_h \subset \textbf{H}^1_0( \Omega ) $, $ M_h \subset L^2_0( \Omega) $ and $ \textbf{W}_h \subset \textbf{H}^1_n( \Omega ) $ be a set of compatible approximation spaces for the velocity, pressure and magnetic field, a fully discrete first-order IMEX SAV scheme is as follows: ($ \textbf{u}_h^{n+1}, p_h^{n+1}, \textbf{b}_h^{n+1} $) in 
($\textbf{X}_h, M_h, \textbf{W}_h$) and $q_h^{n+1} \in \mathbb{R}$ such that
     \begin{eqnarray}
   && ( d_t \textbf{u}_h^{n+1}, \textbf{v}_h ) + \nu (\nabla \textbf{u}_h^{n+1}, \textbf{v}_h ) - (p_h^{n+1}, \nabla\cdot \textbf{v}_h ) =   \alpha \exp( \frac{t^{n+1} }{T} ) q^{n+1}_h \left(  (\nabla \times \textbf{b}^n_h)  \times \textbf{b}^n_h,  \textbf{v}_h   \right)
   \nonumber \\
&& \ \ \ \ \ \ \ \ \ \ 
- \exp( \frac{t^{n+1} }{T} ) q^{n+1}_h ( \textbf{u}^{n}_h \cdot \nabla \textbf{u}^{n}_h ), 
\textbf{v}_h ),   \ \ \forall 
 \textbf{v}_h \in \textbf{X}_h,   \label{e_SAV_Fully discrete_first_u} \\
&&   ( \nabla \cdot \textbf{u}_h^{n+1}, \xi_h)=0, \ \ \forall  \xi_h \in M_h, 
\label{e_SAV_Fully div_u} \\
 &&  (d_t \textbf{b}_h^{n+1}, \textbf{w}_h ) + \eta ( \nabla \times \textbf{b}_h^{n+1} , \nabla \times \textbf{w}_h ) + \eta (\nabla \cdot \textbf{b}_h^{n+1}, \nabla \cdot \textbf{w}_h) \nonumber \\
&& \ \ \ \ \ \ \ \ \ \ 
+ \exp( \frac{t^{n+1} }{T} ) q^{n+1}_h \left( \nabla \times ( \textbf{b}^n_h \times \textbf{u}^n_h ),  \textbf{w}_h \right) = 0,  \ \ \forall \textbf{w}_h \in \textbf{W}_h, 
\label{e_SAV_Fully discrete_first_b} \\   
&&  d_t q^{n+1}_h = -\frac{1}{T}q^{n+1}_h + \exp( \frac{t^{n+1} }{T} ) \nonumber\\
&& \big( (\textbf{u}^n_h \cdot\nabla  \textbf{u}^n_h, \textbf{u}^{n+1}_h ) - \alpha  ( ( \nabla \times \textbf{b}^n_h )  \times \textbf{b}^n_h , \textbf{u}^{n+1}_h )
+ \alpha(  \nabla \times  (\textbf{b}^n_h  \times \textbf{u}^n_h )  , \textbf{b}^{n+1}_h )\big).\label{e_SAV_Fully discrete_first_q} 
     \end{eqnarray} 
 Second-order fully discrete IMEX SAV scheme can be constructed similarly.
 
 Following the same procedure as in the proof of Theorem \ref{thm_energy stability_first order}, namely, setting $\textbf{v}_h=\textbf{u}_h^{n+1}$, $\xi_h=p_h^{n+1}$, $\textbf{w}_h= \alpha \textbf{b}_h^{n+1}$ in \eqref{e_SAV_Fully discrete_first_u}-\eqref{e_SAV_Fully discrete_first_b} respectively and taking the inner product of \eqref{e_SAV_Fully discrete_first_q} with $q_h^{n+1}$, we can obtain the following stability result: The scheme \eqref{e_SAV_Fully discrete_first_u}-\eqref{e_SAV_Fully discrete_first_q} is unconditionally stable in the sense that
 \begin{equation}\label{fully discrete_eng}
\aligned
E_h^{n+1}-E_h^{n} \leq & - \nu\Delta t \| \nabla \textbf{u}_h^{n+1} \|^2 -  \eta \alpha \Delta t \| \nabla \textbf{b}_h^{n+1} \|^2 \\
& -  \eta \alpha \Delta t \| \nabla \times \textbf{b}_h^{n+1} \|^2  -\frac{1}{T}\Delta t|q_h^{n+1}|^2, \ \ \forall \Delta t,\; n\geq 0,
\endaligned
\end{equation} 
where 
\begin{equation*}
E_h^{n+1}=\frac 12 \|\textbf{u}_h^{n+1}\|^2 + \frac \alpha 2 \|\textbf{b}_h^{n+1}\|^2 +\frac 12|q_h^{n+1}|^2 .
\end{equation*}

In our simulation, we use $(P_2,P_1,P_2)$ finite-elements to approximate velocity, pressure and magnetic field, respectively. Note that the $(P_2,P_1)$ finite-elements for velocity and pressure satisfy the inf-sup conditions so that one can easily show that the fully discrete scheme \eqref{e_SAV_Fully discrete_first_u}-\eqref{e_SAV_Fully discrete_first_q} coupled with $(P_2,P_1,P_2)$ finite elements are well posed and can be solved following the procedure described in Section 2.
 
In this example, we set $\Omega=(0,1)\times(0,1)$, $\nu=0.01$, $\eta=0.01$, $\alpha=1$, $T=1$. The right hand side of the equations  is computed according to the analytic solution given as below:\\
\begin{equation*}
\aligned
\begin{cases}
u_1(x,y,t)=\pi k \sin^2(\pi x)\sin(\pi y)\cos(t),\\
u_2(x,y,t)=-\pi k \sin(\pi x)\sin^2(\pi y)\cos(t),\\
p(x,y,t)=k (x-1/2)(y-1/2)\cos(t)/10,\\
b_1(x,y,t)=k \sin(\pi x)\cos(\pi y)\cos(t),\\
b_2(x,y,t)=-k \cos(\pi x)\sin(\pi y)\cos(t),
\end{cases}
\endaligned
\end{equation*}
where $k=0.01$. To test the time accuracy, we choose $h=0.005$  so that the spatial discretization error is negligible compared to the time discretization error for the time steps used in this experiment.

 \begin{table}[htbp]
\renewcommand{\arraystretch}{1.1}
\small
\centering
\caption{Errors and convergence rates with the first-order scheme \eqref{e_SAV_scheme_first_u}-\eqref{e_SAV_scheme_first_q}}\label{table1}
\begin{tabular}{p{1.2cm}p{1.8cm}p{1.3cm}p{1.8cm}p{1.3cm}p{1.8cm}p{1.3cm}}\hline
$\Delta t$  &$\|\mathbf{u}_h-\mathbf{u}\|_{H^1}$ &Order   &$\|\mathbf{u}_h-\mathbf{u}\|_{L^2}$&Order &$\|p_h-p\|_{L^2}$ &Order   \\ \hline
1/2   & 8.26E-3  & ---    &1.34E-3  &---      &2.66E-5  &--- \\    
1/4   & 3.96E-3 &1.06   &7.16E-4  &0.91   &1.16E-5  &1.12 \\
1/8   & 1.93E-3 &1.04   &3.70E-4  &0.95   &5.41E-6  &1.10  \\
1/16 & 9.52E-4 &1.04   &1.89E-4  &0.97   &2.61E-6  &1.05   \\
1/32 &4.72E-4  &1.01   &9.51E-5  &0.99   &1.28E-6  &1.03    \\
1/64 &2.35E-4  &1.01   &4.78E-5  &0.99   &6.33E-7 &1.01   \\ \hline
\end{tabular}
\end{table}

 \begin{table}[htbp]
\renewcommand{\arraystretch}{1.1}
\small
\centering
\caption{Errors and convergence rates with the first-order scheme \eqref{e_SAV_scheme_first_u}-\eqref{e_SAV_scheme_first_q} }\label{table2}
\begin{tabular}{p{1.2cm}p{2.0cm}p{1.5cm}p{2.0cm}p{1.5cm}}\hline
$\Delta t$  &$\|\mathbf{b}_h-\mathbf{b}\|_{H^1}$& Order &$\|\mathbf{b}_h-\mathbf{b}\|_{L^2}$&Order   \\ \hline
1/2   &4.52E-3  & ---    &1.22E-3  & --- \\  
1/4   &2.10E-3  &1.11  &6.39E-4  &0.94 \\
1/8   &1.00E-3  &1.07  &3.27E-4  &0.97  \\
1/16 &4.89E-4  &1.04  &1.65E-4  &0.98   \\
1/32 &2.41E-4  &1.02  &8.31E-5  &0.99    \\
1/64 &1.20E-4  &1.01  &4.17E-5  &1.00   \\ \hline
\end{tabular}
\end{table}

 \begin{table}[htbp]
\renewcommand{\arraystretch}{1.1}
\small
\centering
\caption{Errors and convergence rates with the second-order scheme \eqref{e_SAV_scheme_second_u}-\eqref{e_SAV_scheme_second_q} }\label{table3}
\begin{tabular}{p{1.2cm}p{1.8cm}p{1.3cm}p{1.8cm}p{1.3cm}p{1.8cm}p{1.3cm}}\hline
$\Delta t$  &$\|\mathbf{u}_h-\mathbf{u}\|_{H^1}$ &Order   &$\|\mathbf{u}_h-\mathbf{u}\|_{L^2}$&Order &$\|p_h-p\|_{L^2}$ &Order   \\ \hline
1/2   &6.43E-3 & ---   &8.84E-4   &---   &1.94E-5 & --- \\   
1/4   &1.99E-3 &1.70 &2.32E-4  &1.93 &5.23E-6 &1.89 \\
1/8   &5.49E-4 &1.85 &5.35E-5  &2.12 &1.38E-6 &1.92 \\
1/16 &1.44E-4 &1.93 &1.26E-5  &2.09 &3.53E-7 &1.96  \\
1/32 &3.70E-5 &1.96 &3.05E-6  &2.04 &8.92E-8 &1.99    \\
1/64 &1.03E-5 &1.85 &7.52E-7  &2.02 &2.24E-8 &1.99    \\ \hline
\end{tabular}
\end{table}

 \begin{table}[htbp]
\renewcommand{\arraystretch}{1.1}
\small
\centering
\caption{Errors and convergence rates with the second-order scheme \eqref{e_SAV_scheme_second_u}-\eqref{e_SAV_scheme_second_q} }\label{table4}
\begin{tabular}{p{1.2cm}p{2.0cm}p{1.5cm}p{2.0cm}p{1.5cm}}\hline
$\Delta t$  &$\|\mathbf{b}_h-\mathbf{b}\|_{H^1}$& Order &$\|\mathbf{b}_h-\mathbf{b}\|_{L^2}$&Order   \\ \hline
1/2   &3.54E-3  & ---    &8.38E-4  & --- \\  
1/4   & 1.06E-3 &1.74  &2.30E-4 &1.87 \\
1/8   & 2.90E-4 &1.88  &5.57E-5 &2.05 \\
1/16 & 7.54E-5 &1.94  &1.35E-5 &2.04 \\
1/32 & 1.92E-5 &1.97  &3.32E-6 &2.02 \\
1/64 &4.88E-6  &1.98 &8.23E-7 &2.01   \\ \hline
\end{tabular}
\end{table}

Numerical results for this example  with first- and second-order schemes are presented in Tables \ref{table1}-\ref{table4}. We observe that the results for the first-order scheme \eqref{e_SAV_scheme_first_u}-\eqref{e_SAV_scheme_first_q} are consistent with the error estimates in Theorems  \ref{thm: error_estimate_ubq} and \ref{thm: error_estimate_p}. While second-order convergence rates for the velocity, pressure and magnetic field were observed for the second-order scheme \eqref{e_SAV_scheme_second_u}-\eqref{e_SAV_scheme_second_q}.

\section{Concluding remarks}
We constructed first- and second-order discretization schemes in time based on the SAV approach  for the MHD equations. The nonlinear terms are treated explicitly in our schemes so they only require solving a sequence of linear differential equations with constant coefficients at each time step. Thus, the schemes are efficient and easy to implement.

 Despite the fact that the nonlinear terms are treated explicitly, we proved that our schemes are unconditionally energy stable. This is made possible by introducing a purely artificial scalar auxiliary variable, $q(t)$, which enables  the nonlinear contributions to the energy to cancel with each other as in the continuous case, leading to the  unconditionally energy stability.

By using the unconditional energy result which leads to uniform bound on  the numerical solution , we derived rigorous error estimates for the velocity, pressure and magnetic field of the first-order scheme in the two-dimensional case without any condition on the time step. To the best of our knowledge, this is the first linear, unconditional energy stable and convergent scheme  with fully explicit treatment for the MHD equations. We believe that the error estimates can also be established for the second-order scheme in the two-dimensional case although the process will surely be much more tedious. However, it appear that the error estimates can not be easily extended to the three dimensional case as our proof uses essentially some inequalities which are only valid in  the two-dimensional case.


\bibliographystyle{siamplain}
\bibliography{MHD_SAV}

\end{document}